\documentclass[12pt, oneside]{amsart}   	
\usepackage[margin= 1in]{geometry}   
\usepackage{geometry}                		
\usepackage{graphicx}												
\title[Comparison of Non-Archimedean and Logarithmic Mirror Constructions]{Comparison of Non-Archimedean and Logarithmic Mirror Constructions via the Frobenius Structure Theorem}
\author{Samuel Johnston}
\address{Samuel Johnston, Imperial College London, South Kensington Campus, London SW7 2AZ, UK}
\email{samuel.johnston@ic.ac.uk}
\date{}		

\usepackage{amssymb}
\usepackage{mathrsfs}
\usepackage{amsthm}
\usepackage{amsmath}
\usepackage{tikz-cd}
\usepackage{hyperref}
\usepackage{dsfont}
\usepackage{tikz-3dplot}
\newcounter{thm}
\newtheorem{remark}[thm]{Remark}
\newtheorem{proposition}[thm]{Proposition}
\newtheorem{definition}[thm]{Definition}
\newtheorem{theorem}[thm]{Theorem}
\newtheorem{example}[thm]{Example}
\newtheorem{conjecture}[thm]{Conjecture}
\newtheorem{lemma}[thm]{Lemma}
\newtheorem{corollary}[thm]{Corollary}
\newtheorem{assumption}[thm]{Assumption}
\newcommand{\ZZ} {\mathbb{Z}}
\newcommand{\NN}{\mathbb{N}}
\newcommand{\kk}{\mathds{k}}

\usetikzlibrary{fadings}

\numberwithin{equation}{section}
\numberwithin{thm}{section}

\begin{document}

\begin{abstract} 
For a log Calabi Yau pair ($X,D$) with $X\setminus D$ smooth affine, satisfying either a maximal degeneracy assumption or contains a Zariski dense torus, we prove under the condition that D is the support of a nef divisor that the structure constants defining a trace form on the mirror algebra constructed by Gross-Siebert are given by the naive curve counts defined by Keel-Yu. As a corollary, we deduce the equality of the mirror algebras constructed by Gross-Siebert and Keel-Yu in the case $X\setminus D$ contains a Zariski dense torus. In addition, we use this result to prove a mirror conjecture proposed by Mandel for Fano pairs satisfying the maximal degeneracy assumption.

\end{abstract}

\maketitle

\section{Introduction}
Since the research area was conceived in the 90s, Mirror Symmetry has been explored using a wide range of geometric perspectives. In this paper, we will compare two approaches to the subject which have recently proven fruitful, one using non-archimedean geometry constructed by Keel and Yu in \cite{archmirror}, and the other using logarithmic geometry constructed by Gross and Siebert \cite{int_mirror} and \cite{scatt}. In both approaches, to certain classes of log Calabi-Yau pairs, $(X,D)$, we can construct a mirror family over a formal neighborhood of a point in $Spec\text{ }\kk[NE(X)]$. However, the multiplicative structure of the ring of functions on these families takes into account very different types of data. In the Keel-Yu construction, the algebra is constructed out of data coming from naive counts of non-archimedean analytic disks intersecting $D$ at a finite number of places, while the Gross-Siebert mirror is constructed from virtual counts of log stable maps, which generally will have components lying entirely in the boundary $D$. We will see that in most settings in which both mirror have been constructed, they give the same result.

More precisely, we assume $U$ a smooth connected affine variety with $(X,D)$ a smooth projective log Calabi-Yau compactification with $D:= X\setminus U$ an snc divisor having an irreducible decomposition $D = \sum_i D_i$. We will assume in addition that there exists a nef divisor $\sum_i a_iD_i$ for some $a_i > 0$, and that $X$ satisfies either Assumption $1.1$ of \cite{scatt} or $U$ contains a Zariski open algebraic torus. We will refer to the existence of the nef divisor as a semipositivity assumption. To fix notation, for $p_1,\ldots,p_k \in \mathscr{P} \subset \Sigma(X)$ contact orders in the Kontsevich-Soibelman skeleton of $(X,D)$, and $\textbf{A} \in NE(X)$, denote by $\eta(p_1,\ldots,p_k,\textbf{A})$ the curve count given by Definition $1.1$ of \cite{archmirror}, and $\textbf{P}$ the corresponding intersection profile with a divisor $D$ determined by $p_1,\ldots,p_k,0$. Our main result is the following:

\begin{theorem}\label{mthm11}
Letting $x\in U$ be generic, and $\mathscr{M}(X,\textbf{P},\textbf{A})_x$ the moduli space of log stable maps $((C,x_1,\ldots,x_k,x_{out}),f)$ with contact order at marked points given by $\textbf{P}$, of total curve class $\textbf{A}$, and $f(x_{out}) = x$, we have:

\[\int_{[\mathscr{M}(X,\textbf{P},\textbf{A})_x]^{vir}} \psi_{x_{out}}^{k-2} = \eta(p_1,\ldots p_k,\textbf{A}).\]

\end{theorem}

By the non-degeneracy of the Keel-Yu trace induced by $\eta$, proved in Section $18$ of \cite{archmirror}, together with an expression for the $\vartheta_0$ term of the products of two and three $\vartheta$ functions in terms of the above collection log Gromov-Witten invariants proven in Theorem $9.3$ of \cite{int_mirror}, we show the two mirror constructions of interest in this paper are equal when assuming the semipositivity assumption of Theorem \ref{mthm11}, as we record in the following corollary:

\begin{corollary}\label{mcr11}
For $U$ a smooth affine log Calabi-Yau containing a Zariski dense torus, the Keel-Yu mirror equals the Gross-Siebert mirror for a compactifying pair $(X,D)$ satisfying the condition of Theorem \ref{mthm11}. In particular, the Gross-Siebert mirror extends to a family over $Spec\text{ }\kk[NE(X)]$.
\end{corollary}

By embedding into an appropriate projective space and compactifying, followed by a birational modification supported on the complement of $U$ to ensure $\overline{U}\setminus U$ has simple normal crossings, $U$ always admits a compactification by a pair $(X,D)$ as considered in Theorem \ref{mthm11} and Corollary \ref{mcr11}. 

\begin{remark}\label{assump}
\begin{enumerate}
\item A target which satisfies Assumption $1.1$ of \cite{scatt} does not necessarily satisfy the assumption of Keel and Yu and vice versa. In any case, by the construction of \cite{int_mirror}, there is always a Gross-Sibert mirror associated with any target which admits a Keel-Yu mirror. This is why we only assume a Zariski dense torus in Corollary \ref{mcr11}. Moreover, the naive curve counts $\eta(p_1,\ldots,p_k,\textbf{A})$ are well defined as long as $U$ is affine and $(X,D)$ is log Calabi-Yau. 

\item Given any log Calabi-Yau variety $(X,D)$ satisfying all of the conditions of Theorem \ref{mthm11} except $(X,D)$ is only log smooth and $D$ does not have simple normal crossings, there exists a log blowup of $(\tilde{X},\tilde{D}) \rightarrow (X,D)$ such that $(\tilde{X},\tilde{D})$ does satisfy all of the conditions of Theorem \ref{mthm11}. Furthermore, in the course of proving Theorem \ref{mthm11}, we will see that the invariant does not change after this modification, see Lemma \ref{birinv}. As a result, a modified version of Theorem \ref{mthm11} holds replacing the snc assumption with $(X,D)$ is log smooth. Since the proof of the general theorem reduces to a proof of Theorem \ref{mthm11}, we restrict to the given hypotheses for ease of exposition. 

\item Under the expectation that the Gross-Siebert mirror algebra is computing degree $0$ symplectic cohomology for the affine variety $U$, the intrinsic mirror construction should satisfy a form of birational invariance. Moreover, the collection of compactifictions of $U$ satisfying the conditions of the theorem are cofinal in the inverse system of all projective compactifications of $U$. Combining the expectation with the fact above, the equality of Corollary \ref{mcr11} is expected to hold in all cases both mirrors can be constructed. 
\end{enumerate}
\end{remark}

We will keep in mind the following running example throughout this paper:
\begin{example}

Let $X = \mathbb{P}^2$, and $D = D_1 + D_2$ with $D_2$ a smooth conic and $D_1$ a line not tangent to $D_2$. Then $D$ is an ample snc anti-canonical divisor, so $(X,D)$ is log Calabi-Yau pair with $X\setminus D$ smooth affine. 
\end{example}

Additionally, we prove the weak Frobenius structure theorem for log Calabi-Yau pairs $(X,D)$ satisfying Assumption $1.1$ of \cite{scatt}, generalizing Theorem $9.3$ of \cite{int_mirror} when $(X,D)$ satisfies Assumption $1.1$.

\begin{theorem}[\cite{int_mirror} Conjecture $9.2$]\label{mthm21}
Let $(X,D)$ be a log Calabi-Yau pair satisfying Assumption $1.1$ of \cite{scatt}. Denote by $R_{(X,D)}$ the resulting mirror algebra, $\vartheta_{p_1} \ldots \vartheta_{p_k} \in R_{(X,D)}$ theta functions associated with contact orders $p_1,\ldots,p_k \in \mathscr{P}(\ZZ)$, and $\langle \vartheta_{p_1}\cdots \vartheta_{p_k} \rangle$ the $\vartheta_0$ coefficient of the product $\vartheta_{p_1}\cdots\vartheta_{p_k} \in R_{(X,D)}$. Letting $\underline{\beta}$ be the undecorated global tropical type with a single vertex, $x \in X\setminus D$ generic, and contact orders $p_1,\ldots,p_k,0$, then

\[\sum_{\beta = (\textbf{A},\underline{\beta})} t^{\textbf{A}}\int_{[\mathscr{M}(X,\beta)_x]^{vir}} \psi_{x_{out}}^{k-2} = \langle \vartheta_{p_1}\cdots \vartheta_{p_k}\rangle \]
where the above sum is over all decorations of $\underline{\beta}$ by curve classes $\textbf{A} \in NE(X)$.
\end{theorem}

Using Theorems \ref{mthm11} and \ref{mthm21}, after restricting further to the case of a target $(X,D)$ a pair with $X$ Fano and $D \in |-K_X|$,  we prove a conjecture of Travis Mandel in \cite{fanoperiod} for a large class of Fano pairs $(X,D)$. Specifically, Theorems \ref{mthm21} and \ref{mthm11} in the Fano case imply Conjecture $1.4$ of \cite{fanoperiod}. Hence, as a corollary of \cite{fanoperiod} Theorem $1.12$, we have the following mirror theorem for Fano varieties which have a canonical wall structure as constructed in \cite{scatt}. 
 
\begin{corollary}\label{mcr2}
Let $(X,D)$ be an snc Fano pair satisfying Assumption $1.1$ of \cite{scatt}. Furthermore, let $D_1,\ldots, D_s$ be the irreducible components of the boundary $X \setminus U$, and define $W = \vartheta_{D_1} + \cdots + \vartheta_{D_s} \in R_{(X,D)}$. Then in the notation of \cite{fanoperiod}, with $\hat{G}_X$ the regularized quantum period of $X$:
\[\hat{G}_X := \sum_{\textbf{A}\in NE(X)} (-K_X\cdot A)!(\int_{[\mathscr{M}_{0,1}(X,\textbf{A})]^{vir}} ev^*([pt])\psi^{((-K_X\cdot A)-2)})t^{\textbf{A}} \in \kk[[NE(X)]]\]
and $\pi_W$ the classical period of $W$:
\[\pi_W:= \sum_{d\ge 0} W^d[\vartheta_0] \in \kk[[NE(X)]]\]
we have the equality:
\[\hat{G}_X = \pi_W.\]
\end{corollary}

We will begin by recalling the relationship between tropical and logarithmic moduli problems, followed by briefly recalling the mirrors constructed using non-archimedean and logarithmic geometry. By making use of the Frobenius structure theorem proven by Keel-Yu for their mirror construction, we will reduce the question of comparing the mirrors to proving Theorem \ref{mthm11}. To prove this theorem, we make use of the restrictions of the targets imposed either by the Zariski dense torus assumption, or those imposed for the construction of the canonical wall structure in \cite{scatt} to remove the toric model assumptions used elsewhere. More precisely, the assumptions allows us to define a useful class of tropical types of punctured log curve called broken line types. Moreover, by appropriately degenerating the point constraint of the log Gromov-Witten invariant of interest, we see that the tropical types of curves contributing to the invariant arise as limits of families of expansions of $k$ broken lines meeting at a point. We then apply a tropical lemma, together with the non-negative boundary condition, to sufficiently restrict curves potentially contributing to the log Gromov Witten invariant with generic point constraint so that they all lie in the smooth locus of the moduli space of interest, and results from \cite{archmirror} Section $3$ will allow us to conclude Theorem \ref{mthm11}. 

Following the proof of Theorem \ref{mthm11} and its corollaries, we explicitly define a tropical type of punctured log curve, $k$-trace type, which appears in the proof of Theorem \ref{mthm11}. By slight modifications of arguments appearing in section $6$ of \cite{scatt}, we relate virtual integrals over moduli spaces of such curves with the trace of $k$ theta functions in the Gross-Siebert mirror algebra. After restricting our targets further to Fano pairs satisfying Assumption $1.1$ of \cite{scatt}, we use the results above to deduce Corollary \ref{mcr2}.

\subsection{Related Work}

Mandel in \cite{clusterGW} has previously proven the full Frobenius structure conjecture for all cluster varieties. In these situations, the target geometries can be tamed by considering toric models associated with various seeds of the cluster variety, corresponding to different choices of dense torus. This allows for tropical curve counting techniques to be utilized in the study of the enumerative geometry of cluster varieties. Instead of assuming a toric model, we use the maximal degeneracy assumption on the boundary to facilitate the connection with tropical geometry. Note however that the non-degeneracy of the trace form is not known in general without the existence of a dense torus.

In work of Tseng and You \cite{RQC_no_log}, they define negative contact order using orbifold Gromov Witten invariants, use it to define a relative quantum cohomology ring, and use this to define another candidate mirror algebra. Moreover, in \cite{Orbi_GW_mirror}, they prove a quantum-classical period correspondence in the same form as Corollary \ref{mcr2} above, after replacing the classical period associated to the superpotential used in this paper with a generating function of analogous orbifold Gromov-Witten invariants, which can be derived in the same manner from the orbifold mirror algebra. This correspondence is proven after imposing a further condition on the boundary to ensure the mirror map is trivial, i.e., the relevant I function equals the relevant J function after no change of variables. In particular, each component of the boundary is assumed to be nef. In this context therefore, we have equalities between certain generating functions of regularized quantum periods, orbifold Gromov-Witten, and log Gromov-Witten invariants. In this paper, we assume only that the Fano pair satisfies Assumption $1.1$ of \cite{scatt}. It is a further interesting question to understand what relationship these invariants have when the mirror map is not assumed to be trivial, and more generally to compare these two candidate mirror algebras.

\subsection{Acknowledgments}
I would like to thank my supervisior Mark Gross for suggesting this problem, and for numerous suggestions and comments throughout the development and writing of this paper, specifically to pursue the perspective inspired by the Frobenius Structure conjecture. I also thank the referee whose thorough comments helped improve the exposition of the paper. The work was funded by a stipend from University of Cambridge Department of Pure Mathematics and Mathematical Statistics, and the ERC Advanced Grant MSAG.

\section{Tropical Types of Log Curves and their Moduli}

Throughout, we let $\kk$ be an algebraically complete field of characteristic $0$. We begin by describing the tropical analogues of the moduli spaces of interest. Let \textbf{Cones} be the category whose objects are rational polyhedral cones $\sigma$ equipped with an integral lattice $\Lambda_\sigma$, i.e, $\sigma \subset \Lambda_{\sigma} \otimes_{\mathbb{Z}} \mathbb{R} =\sigma^{gp}$. Throughout this paper, we refer to $\sigma$ as a real cone, and $int(\sigma)$ as the interior of the cone $\sigma$ i.e. the set of points not contained in proper faces of $\sigma$. Additionally, we let $\sigma_\mathbb{N} = \sigma \cap \Lambda_{\sigma}$ be the monoid of integral points. We let \textbf{RPCC} be the category of rational polyhedral cone complexes. 

By the main theorem of \cite{fun_trop} and later in \cite{punc} Appendix C in the setting of fine log schemes, there exists a functorial tropicalization functor $\Sigma: Log_{f} \rightarrow \textbf{RPCC}$ from the category of fine log schemes to the category of rational polyhedral cone complexes  In the case that $X$ is a smooth scheme over $Spec\text{ }\kk$ and $D$ has simple normal crossings, $\Sigma(X)$ is simply the dual complex of the boundary divisor designated by the log structure on $X$. When $X$ is a toric variety, $\Sigma(X)$ is the underlying polyhedral cone complex of the fan of $X$. Note however that the log structure on $X$ alone does not give an embedding of $\Sigma(X)$ into a real vector space, i.e., $\Sigma(X)$ does not have a canonical affine structure. We denote the integral points of the cone complex $\Sigma(X)$ by $\Sigma(X)(\mathbb{Z})$. We denote by $\Sigma(X)^{max}$ the set of maximal cones, and more generally $\Sigma(X)^{[k]}$ for the set of dimension $k$ cones of $\Sigma(X)$. We note that by definition of $\Sigma(X)$, each component $D_i \subset D$ which is Cartier defines a linear function on the tangent space $\sigma^{gp}$ for $\sigma \in \Sigma(X)$. Indeed, there is a unique section $s_i \in \Gamma(X,\overline{\mathcal{M}}_{X}) = \Gamma(X,\mathcal{M}_X/\mathcal{O}^\times_X)$ such that for $q: \mathcal{M}_X \rightarrow \overline{\mathcal{M}}_X$ the quotient map, $q^{-1}(s_i)$ is the $\mathbb{G}_m$-torsor associated with the the divisor $-D_i$, and a section of $\overline{\mathcal{M}}_X$ by definition defines a PL functions on $\Sigma(X)$. For $u \in \sigma^{gp}$, we denote the pairing of $u$ with the previous PL function by $u(D_i)$.

In the interest of using tropicalization to fix additional discrete data in log moduli problems, we recall the notion of an abstract tropical type, see \cite{decomp} Definition $2.19$:

\begin{definition}
An abstract tropical curve over a cone $\sigma \in \textbf{Cones}$ is the data $(G,\textbf{g},l)$ with
\begin{enumerate}
\item $G = G_\sigma$ is a graph with legs, with set of vertices, edges and legs denoted by $V(G_\sigma)$, $E(G_\sigma)$ and $L(G_\sigma)$ respectively.
\item $\textbf{g}: V(G) \rightarrow \NN$ a function called the genus function.
\item $l: E(G) \rightarrow Hom(\sigma_\NN,\NN)\setminus 0$ and assignment of edge lengths. 
\end{enumerate}
\end{definition}

This determines a cone complex $\Gamma_\sigma$ with a map $\pi_{(G,l)}: \Gamma_\sigma \rightarrow \sigma$ satisfying the condition that for $s \in int(\sigma)$, the fiber $\pi_{(G,l)}^{-1}(s) = \Gamma_s$ is a metrized graph whose underlying graph is $G$, and such that the length of $e \in E(G)$ is given by $l(e)(s)$. Moreover, there is a bijection between cones of $\Gamma_\sigma$ which surject onto $\sigma$ and elements of $F(G):= V(G)\cup E(G) \cup L(G)$. For $x \in F(G)$, we denote by $\sigma_x$ the corresponding cone. In particular, for a vertex $v\in V(G_\sigma)$, there is a section $\sigma \rightarrow \Gamma_\sigma$ of the projection map which maps isomorphically onto the cone $\sigma_v$.

Given a log smooth scheme $X$ as before, a family of tropical maps with target $\Sigma(X)$ over $\sigma$ as above is a morphism of cone complexes $h: \Gamma_\sigma \rightarrow \Sigma(X)$, with $\Gamma_{\sigma}$ a family of tropical curves. As before, over a point $s \in int(\sigma)$, we have a tropical curve $\Gamma_s$, and a piecewise linear map $h_s\colon \Gamma_s \rightarrow \Sigma(X)$. In particular, the map defines two function:
\begin{enumerate}
\item $\pmb\sigma: V(G) \cup E(G) \cup L(G) \rightarrow \Sigma(X)$ which assigns to a feature of $G$ the smallest cone of the complex $\Sigma(X)$ which the cone associated to the feature maps into.
\item $\textbf{u}(e) \in \pmb{\sigma}(e)^{gp}_{\NN}$ for $e \in E(G) \cup L(G)$ which gives the integral tangent vector given by the image of the tangent vector $(0,1) \in \sigma^{gp}_{e,\NN} = \sigma^{gp}_\NN \oplus \mathbb{Z}$ under the map $h: \Gamma_\sigma \rightarrow \Sigma(X)$.
\end{enumerate}

The data $(G,\pmb\sigma,\textbf{u},g)$ determines the tropical type of the family. Throughout this paper, we will exclusively work in genus $0$, so the genus function will always be the $0$ function, and we will suppress this in the notation from now on. Note that to each tropical type $\sigma$, we have the group $Aut(\sigma)$, whose elements are automorphisms of the graph $G_{\sigma}$ preserving the data of target cones, contact orders, and legs.

Given a family of tropical curves $\Gamma_\sigma \rightarrow \sigma$, we also consider the tropical analogue of punctured log curves. Roughly speaking, these are families of tropical curves with possibly bounded legs. The resulting metrized graphs will at times be referred to as tropical disks:

\begin{definition}
Given a family of ordinary tropical curves $\pi_{(G,l)}: \Gamma_\sigma \rightarrow \sigma$, \emph{a puncturing} $\Gamma_\sigma^\circ \rightarrow \sigma$ \emph{of} $\Gamma_\sigma$ is an inclusion of a sub cone complexes $\Gamma_{\sigma}^\circ \rightarrow \Gamma_\sigma$ over $\sigma$ which is an isomorphism away from cones of $\Gamma_\sigma$ associated with legs of $G$. Moreover, for cones $\sigma_l \in \Gamma_\sigma$ associated with punctured legs of $G$, there exists a unique cone $\sigma_l^\circ \in \Gamma_{\sigma}^\circ$ intersecting the interior of $\sigma_l$, and there exists a section $\sigma \rightarrow \sigma_l^\circ$ which when composed with $\sigma_l \rightarrow \Gamma_\sigma$ is a section associated with the unique vertex of $G$ contained in the leg $l$. 
\end{definition}

Note that the definition of a tropical type of a tropical map generalizes easily to the definition of a punctured tropical map. Given the data of a tropical type $(G,\pmb\sigma,\textbf{u})$ above, we say the tropical type is realizable if there exists a family of tropical curves realizing this type. In such a situation, there is a cone $\tau \in \textbf{Cones}$ which is the universal family of tropical curves of type $\tau$, i.e, there exists a universal family $\Gamma_\tau/\tau$ of tropical maps to $\Sigma(X)$ with type $(G,\sigma,\textbf{u})$, and all other families of such tropical maps over a cone $\omega \in \textbf{Cones}$ arise as pullbacks along an appropriate map $\omega \rightarrow \tau$. Whenever the type is realizable, we will refer to both the tropical type and the associated universal cone by the same name. 

Recall now the notion of a punctured log curve, defined in \cite{punc}. For our purposes, we will mostly think of them as ordinary log curves $C$, equipped with an enlarged fine log structure such that after applying functorial tropicalization to the induced morphisms of log schemes, we get a puncturing of $\Sigma(C)$. Most examples of families of tropical curves considered in this paper come from tropicalizing a family of punctured log curve $C \rightarrow S$ with target a fine and saturated log smooth log scheme $X$, or its Artin fan $\mathcal{X}$. More precisely, given a diagram in fine log schemes:

\[\begin{tikzcd}
C \arrow[dd,bend right  =40]  \\
C^\circ \arrow{u} \arrow{r} \arrow{d} & X \\
S
\end{tikzcd}\]
with $C^\circ \rightarrow S$ the puncturing of the family of log curve $C \rightarrow S$, we apply functorial tropicalization to produce the following diagram in the category of generalized cone complexes:

\[\begin{tikzcd}
\Sigma(C) \arrow[dd,bend right = 40]  \\
\Sigma(C^\circ) \arrow{u} \arrow{r} \arrow{d} & \Sigma(X) \\
\Sigma(S)
\end{tikzcd}\]

As observed in \cite{punc} Section $2.2$, $\Sigma(C)$ is a tropical curve parameterized by the generalized cone complex $\Sigma(S)$, and $\Sigma(C^\circ)$ is a puncturing of $\Sigma(C)$. For each cone $\sigma \in \Sigma(S)$, we have an induced family of tropical maps $\Gamma_\sigma^\circ \rightarrow \Sigma(X)$ of a fixed tropical type $\tau_\sigma$. The cone $\sigma \in \Sigma(S)$ is associated with a stratum $S_\sigma \subset S$, and we say the family of punctured log maps restricted to $S_\sigma$ has tropical type $\tau_{\sigma}$.

By \cite{punc}, to a tropical type $\tau$, there is an associated DM log stack $\mathscr{M}(X,\tau)$ of log stable maps marked by $\tau$ in the sense of Definition $3.7$ of \cite{punc}. When $\tau$ is realizable, this condition on a geometric point $(C^\circ,f)$ is equivalent to the condition that the tropical type of $(C^\circ,f)$ has an associated universal tropical moduli space $\tau'$ which contains a face $\tau \subset \tau'$, such that the restriction of the universal family $\Gamma_{\tau'} \rightarrow \tau'$ to the face $\tau$ is the universal family of the realizable tropical type $\tau$. The moduli stack $\mathscr{M}(X,\tau)$ in general possesses the typical singular behavior that is standard in a moduli stack of maps, and in order to construct a perfect obstruction theory which allows for the definition of log Gromov Witten invariants, we make further use of the tropical theory.

Associated to a fine and saturated log stack $X$, we have the tautological morphism $X \rightarrow \textbf{Log}$ which factors through a unique log algebraic stack $X \rightarrow \mathcal{X}$, with $\mathcal{X} \rightarrow \textbf{Log}$ representable and \'etale. We call the log stack $\mathcal{X}$ the Artin fan of $X$. See \cite{bound} section $3$ for further details. Specializing to the case in which $X$ is a log smooth scheme, the moduli problem over schemes of basic families of punctured log maps to $\mathcal{X}$ defines an algebraic log stack $\mathfrak{M}(\mathcal{X},\tau)$ of prestable punctured curves equipped with a marking by $\tau$. This log stack is equipped with an idealized log structure described in \cite{punc} Definition $3.22$. We have a morphism $\mathscr{M}(X,\tau) \rightarrow \mathfrak{M}(\mathcal{X},\tau)$ given by composing a punctured log curve $C \rightarrow X$ with $X \rightarrow \mathcal{X}$. Moreover, unlike $\mathscr{M}(X,\tau)$ in general, $\mathfrak{M}(X,\tau)$ is idealized log smooth by \cite{punc} Theorem $3.24$. Furthermore, there is an obstruction theory associated to the morphism $\mathscr{M}(X,\tau) \rightarrow \mathfrak{M}(\mathcal{X},\tau)$, which facilitates virtual pullback in the sense of \cite{vpull}.

In this paper, we will need to impose point constraints at marked points. Ordinarily, evaluation morphisms $\mathscr{M}(X,\beta) \rightarrow \underline{X}$ on underlying stacks do not lift to morphisms of log stacks. For our purposes however, we will only need to impose point constraints at contact order $0$ points, which mitigates many of these difficulties. Indeed, by \cite{int_mirror} Proposition $3.3$, we have a morphism of log stacks $\mathscr{M}(X,\beta) \rightarrow \mathscr{P}(X,0) = X \times B\mathbb{G}_m$, with $B\mathbb{G}_m$ equipped with the trivial log structure. The projection map on underlying stacks lifts to a map of log stacks, and thus equips $\mathscr{M}(X,\beta) \rightarrow \underline{X}$ with the structure of a log map. 

In order to modify our moduli spaces to impose point constraints, we first enhance the moduli stack $\mathfrak{M}(\mathcal{X},\tau)$ by setting $\mathfrak{M}^{ev}(\mathcal{X},\tau) := \mathfrak{M}(\mathcal{X},\tau) \times_{\underline{\mathcal{X}}} \underline{X}$, where $\mathfrak{M}(\mathcal{X},\tau) \rightarrow \underline{\mathcal{X}}$ is defined by evaluation map at $x_{out}$, and $\underline{X} \rightarrow \underline{\mathcal{X}}$ is the underlying morphism from $X$ to its Artin fan. There exists a morphism $\mathscr{M}(X,\tau) \rightarrow \mathfrak{M}^{ev}(\mathcal{X},\tau)$ factoring $\mathscr{M}(X,\tau) \rightarrow \mathfrak{M}(\mathcal{X},\tau)$. In addition, there is a triple of perfect obstruction theories associated to the three morphisms of:

\[\mathscr{M}(X,\tau) \rightarrow \mathfrak{M}^{ev}(\mathcal{X},\tau) \rightarrow \mathfrak{M}(\mathcal{X},\tau)\]
which form a compatible triple. We will primarily be interested in the obstruction theory associated to the left most morphism, and in this case, by \cite{punc} Proposition $4.4$, the perfect obstruction theory in $D^b(\mathscr{M}(X,\tau))$ is given by $R\pi_*(f^*T_X^{log}(-x_{out}))^{\vee}$, where $f$ the universal map $\pi: \mathfrak{C} \rightarrow \mathscr{M}(X,\tau)$. Similar to the case $\mathscr{M}(X,\beta)$, we have an evaluation morphism $ev: \mathfrak{M}^{ev}(\mathcal{X},\beta) \rightarrow X$ at contact order $0$ points. Using these enhanced moduli stacks, given a log morphism $W \rightarrow X$, we take a fine and saturated fiber product $\mathfrak{M}^{ev}(\mathcal{X},\tau)_W := \mathfrak{M}^{ev}(\mathcal{X},\tau) \times^{fs}_X W$ to produce a point constrained moduli space. When $W$ is a trivial log point mapping to $X\setminus D$, this is the same as the ordinary fiber product. Similarly we define $\mathscr{M}(X,\tau)_W := \mathscr{M}(X,\tau) \times^{fs}_X W$. We have a morphism of log stacks $\mathscr{M}(X,\tau)_W \rightarrow \mathfrak{M}^{ev}(\mathcal{X},\tau)_W$ pulled back from the natural morphism $\mathscr{M}(X,\tau) \rightarrow \mathfrak{M}^{ev}(\mathcal{X},\tau)$, and we equip the former morphism with a perfect obstruction theory pulled back from that of the latter morphism. 
 
 \begin{remark}
We note that the algebraic stack $\mathfrak{M}^{ev}(\mathcal{X},\tau)$ is not particularly well behaved. In particular, it is not of finite type, and contains an infinite collection of log strata. This is typically resolved by replacing $\mathfrak{M}^{ev}(\mathcal{X},\tau)$ with the complement of all log strata whose closure does not intersect the image of $\mathscr{M}(X,\tau)$. The latter space is of finite type, as seen for instance by \cite{punc} Theorem $3.11$. One of the essential ingredients missing from $\mathfrak{M}^{ev}(\mathcal{X},\tau)$ is the combinatorial restriction imposed by the logarithmic balancing condition,  intersection numbers with components of the boundary divisor. 
\end{remark} 

These tropical types may also be upgraded to decorated tropical types, by specifying the curve classes associated to vertices in $G$. We will denote decorated tropical type with underlying tropical type $\tau$ by $\pmb\tau$. In the special case in which $G$ has a single vertex, $k$ marked points with contact orders $(a_i)$, and a decoration of the single vertex with a curve class, we call the associated decorated tropical type a tropical class, and we denote the resulting tropical class by $\beta$.

\section{Constructions of the mirror}
In this section, we briefly recall the mirror constructions of Gross-Siebert and Keel-Yu in \cite{int_mirror} and \cite{archmirror}. As was mentioned in Remark \ref{assump}, the logarthimic construction given in \cite{int_mirror} applies to a larger class of log Calabi-Yau targets than the non-archimedean construction given in \cite{archmirror}. However, a different set of assumptions are imposed in \cite{scatt} to construct the canonical wall structure for $(X,D)$. We recall the assumptions needed for these constructions, and their implications in what follows.

Let $U$ be a smooth $n$--dimensional connected affine variety, $X$ a simple normal crossings projective compactification of $U$, and equip $X$ with the divisorial log structure associated to $D := X\setminus U$. Letting $D = \sum_i D_i$ be a decomposition of $D$ into irreducible components, we further require the pair $(X,D)$ to be log Calabi-Yau, which for us will mean one of the following:

\begin{enumerate}
\item $U$ contains a Zariski dense torus $\mathbb{T}^n \subset U$, and the standard volume form on $\mathbb{T}^n$ extends to a non-vanishing holomorphic volume form $\omega \in \Omega(U)$. 
\item $K_X + D =_{\mathbb{Q}} \sum_i a_iD_i$, with $a_i \ge 0$ for all $i$.
\end{enumerate}

The latter condition is implied by the former condition. Indeed, $X$ is a compactification of the torus $\mathbb{T}^n$, hence by \cite{zigzag} there is a zigzag of simple blowups relating $X$ to a toric variety $X_t$ compactifying $\mathbb{T}^n$. The variety $X_t$ has the property that the volume form $\omega$ has at worst simple poles along $X_t\setminus \mathbb{T}^n$, and this condition is maintained for all simple blowups or blowdowns supported on the compliment of $\mathbb{T}^n$. The assumption that $\omega$ extends to a non-vanishing volume form on all of $U$ then implies the existence of a decomposition of $K_X+D$ in the form required for Condition $2$.

 Condition $1$ is assumed for the non-archimedean construction of Keel-Yu, and Condition $2$ is assumed for the logarithmic construction of Gross-Siebert. We make the following assumption throughout:
\begin{assumption}\label{assump}
We assume $U$ is log Calabi-Yau, and either contains a Zariski dense torus, or satisfies the following maximal degeneracy condition, see Assumption $1.1$ of \cite{scatt}:
\begin{enumerate}
\item $\mathscr{P}$ contains an $n$-dimensional cone.
\item Whenever $\rho \in \mathscr{P}^{[n-1]}$, $\sigma \in \Sigma(X)^{max}$ and $\rho \subset \sigma$, we have $\sigma \in \mathscr{P}^{max}$.
\item For $X_{\rho}$ a good strata with $dim\text{ }X_{\rho} > 1$, then $\partial X_{\rho}$ is connected. 
\end{enumerate}
\end{assumption}

While the Gross-Siebert mirror construction applies in greater generality, it will also in general have much less controlled behavior than in our context of interest. If $U$ contains a dense torus and an extension $\omega$ of the canonical volume form on the torus, we call a component of the boundary $D_i \subset D := X \setminus U$ \emph{good} if $\omega$ has a pole along $D_i$. In the other case, we say the component is good if after making the choice of decomposition, we have $a_i = 0$ in the decomposition of $K_X+D$. In any case, a stratum of $(X,D)$ is good if all components $D_i$ of the boundary it is contained in are good. We define the pair $(B,\mathscr{P})$ to be the subcomplex $\mathscr{P} \subset \Sigma(X)$ containing cones associated to all good strata, and $B$ the topological realization of the associated cone complex. Note that if $(X,D)$ is minimal log Calabi-Yau, i.e., $K_X + D = 0$, then $(B,\mathscr{P})$ is simply $\Sigma(X)$.

 If we assume that $(X,D)$ satisfies Assumption $1.1$ of \cite{scatt}, the $0$ and $1$ dimensional strata obtained as intersections of good divisors are isomorphic as log schemes to $0$ and $1$ dimensional strata in appropriate toric varieties $X_t$. Using the the affine structure on $M_\mathbb{R}$, we may give $B$ the structure of an integral affine manifold across codimension $0$ and $1$ cones, with integer points given by $B(\mathbb{Z}) \subset \Sigma(X)(\mathbb{Z})$. We call the union of such cones $B_0$. This gives $B$ the structure of a pseudo-manifold, with singular locus contained in the union of codimension $2$ cones, see \cite{scatt} Section $1.3$ for details. For example, in the case that $U = \mathbb{T}^n$, and $X$ is a toric variety, then this integral affine structure extends over all of $\Sigma(X)$ and is equal to the standard integral affine structure given by embedding the fan of $X$ into $\mathbb{R}^n$. 

Without making Assumption $1.1$ and instead assuming that $U$ contains a dense torus, then by \cite{archmirror} Lemma $2.2$ and $2.6$,  after an appropriate toric blowup $\tilde{X} \rightarrow X$ of the target induced by a certain subdivision $(\tilde{B},\widetilde{\mathscr{P}}) \rightarrow (B,\mathscr{P})$ there exists a toric variety $X_t$ with fan $\Sigma_t$ such that there is a homeomorphism $(\tilde{B},\widetilde{\mathscr{P}}) \cong (M_{\mathbb{R}},\Sigma_t)$, and the isomorphism locus of the rational morphism $\tilde{X} \dashrightarrow X_t$ contains the generic points of each stratum associated with a cone of $\widetilde{\mathscr{P}}$. In particular, $(\tilde{X},\tilde{D})$ satisfies conditions $(1)$ and $(3)$ of Assumption $1.1$, and all $1$ strata are isomorphic to $\mathbb{P}^1$ and contain $2$ good zero-dimensional strata. It is straightforward to see that these conditions must then also be true for $(X,D)$. In general however, $(X,D)$ will not satisfy condition $(2)$. To handle this case, let $X_{good}$ be the log scheme corresponding to $X$ equipped with the divisorial log structure associated with the good divisors $D_{good}$. By Lemma $1.9$ of \cite{scatt}, the dimension $1$ strata of $X_{good}$ are isomorphic as log schemes to dimension $1$ strata of an appropriate toric varieties $X_{t_i}$, which allows us to define an integral affine structure on $B_0$ as above. Note however that we do not necessarily have $X_t = X_{t_i}$. We capsulize the property above in the following lemma:

\begin{lemma}\label{1strat}
When $U := X\setminus D$ contains a Zariski dense algebraic torus, then $\Sigma(X^{good}) \cong \Sigma(X_t)$ for a toric variety $X_t$, and each $1$ stratum of the log scheme $X^{good}$ is isomorphic as a log scheme to a $1$ stratum of a toric variety.
\end{lemma}

 To utilize this structure to understand our original log scheme, we note there is a morphism of log schemes $X \rightarrow X_{good}$ which is the identity on underlying schemes, since each is equipped with divisorial log structures, and the divisor associated with $X_{good}$ is contained in the divisor associated with $X$. 
 
\begin{remark}
\begin{enumerate}
\item Note that Theorem \ref{mthm11} does not need to assume that the full non-archimedean mirror exists for $(X,D)$. Indeed, the naive curve counts $\eta(\textbf{P},\textbf{A})$ exist assuming only that $U$ is affine log Calabi-Yau. 
\item The integral affine structure described in the case $X$ satisfies Assumption $1.1$ of \cite{scatt} is similar to an integral affine structure which appeared in \cite{SYZarch}, which is induced by the  non-archimedean SYZ fibration explored in loc. cit.
\end{enumerate}
\end{remark}

\begin{example}\label{runex}
Consider $X = (\mathbb{P}^2,D_1+D_2)$ with $D_1$ a line and $D_2$ a conic in general positions. This is a log Calabi-Yau variety, satisfying the conditions of Theorem \ref{mthm11}. In this case $K_X + D = 0$, so $\mathscr{P} = \Sigma(X)$. Then $\Sigma(X)$ consists of two maximal two dimensional cones, glued along zero and one dimensional faces. We denote by $\rho_1$ and $\rho_2$ the one-dimensional cones in $\Sigma(X)$ associated to $D_1$ and $D_2$ respectively, and $\sigma_1$, $\sigma_2$ some labeling of the two maximal cones of $\Sigma(X)$ associated to the intersection points of $D_1$ and $D_2$. We denote primitive generators of the rays by $v_1$ and $v_2$, and elements in $\sigma_i(\mathbb{Z})$ by $av_{1,i} + bv_{2,i}$. 
\end{example}

From the data of the log Calab-Yau pair $(X,D)$ above, we construct the mirror algebra by first extracting the following data: Let $NE(X)\subset N_1(X)$ be the monoid of numerically effective curves, $L$ an ample line bundle on $X$, and set $S_X$ and $\hat{S}_X$ to be $\mathbb{Z}[NE(X)]$ and  $\mathbb{Z}[[NE(X)]]$ respectively, with the latter ring being complete with respect to the filtration induced by intersection pairing with $L$, i.e. by the monomial ideals $L_c$ generated by curve classes $\pmb A$ such that $\pmb A \cdot L \ge c$. Note that for any $c>0$, there are only finitely many curve classes in $\pmb A \in NE(X)$ such that $\pmb A \cdot L < c$. 

Now let $R = \bigoplus_{p \in B(\mathbb{Z})} \vartheta_p\cdot S_X$ be the free $S_X$-module generated by formal sums of ``theta functions" $\vartheta_p$. The mirror algebra in both cases will be an $S_X$-algebra, with $S_X$-module structure given as above. Note in particular that since the underlying $S_X$-module structure is free, the resulting family defined over $Spec$ $S_X$ will be flat. What remains to be defined, and where the approaches diverge, in implementation if not in spirit, is how to define an appropriate multiplication rule. To define the product, it suffices to define the product on the free generators of $R_{(X,D)}$. Thus, for $p,q \in B(\mathbb{Z})$, we write:

\begin{equation}\label{stcont}
\vartheta_p \cdot \vartheta_q = \sum_{r \in B(\mathbb{Z})}\sum_{\textbf{A} \in NE(X)} \alpha_{p,q,\textbf{A}}^rt^\textbf{A}\vartheta_r.
\end{equation}

In both works, the structure constants are certain weighted counts of tropical disks. To describe a contributing disk, we consider tropical disks in $B$ with two unbounded legs with outgoing slope $p$ and $q$, identifying an element of $B(\mathbb{Z})$ with the corresponding vector, and bounded leg with outgoing derivative $-r$ which contains $r \in B(\mathbb{Z})$. In the context of non-archimedean geometry, these tropical curves correspond to the skeletons of non-archimedean analytic disks which map to the Berkovich analytification $X^{an}$, while in the log geometric setting, these correspond to a type of punctured log map. 

\begin{figure}[h]
\centering
\begin{tikzpicture}
\fill[white!70!blue, path fading = north] (0,0)--(-3,0)--(-3,3)--(0,3)--cycle;
\fill[white!70!blue, path fading = north] (0,0)--(3,0)--(3,3)--(0,3)--cycle;
\fill[white!70!blue, path fading = south] (0,0)--(3,-3)--(3,0)--cycle;
\fill[white!70!blue, path fading = south] (0,0)--(3,-3)--(-3,-3)--(-3,0)--cycle;
\draw[black] (0,0)--(0,3);
\draw[black] (0,0)--(3,0);
\draw[black] (0,0)--(3,-3);
\draw[black] (0,0)--(-3,0);
\draw[black] (0,0)--(3,3);
\draw[ball color = red] (1.5,0) circle (0.5mm);
\draw[ball color = red] (1.5,1.5) circle (0.5mm);
\draw[-,color = red] (1.5,0)--(1.5,1.5);
\draw[-,color = red] (0,0)--(1.5,1.5);
\draw[ball color = red] (0,0) circle (0.5mm);
\draw[-,color = green] (0,0)--(1.5,0);
\draw[->,color = red] (1.5,1.5)--(2.25,3);
\draw[->,color = red] (1.5,0)--(.2,0);
\draw[->,color = red] (1.5,0)--(3,-1.5);
\end{tikzpicture}
\caption{An example of a tropical curve contributing to a a structure constant, colored in red. The legs with infinite extent correspond to the ``input" contact orders $p$ and $q$, and the bounded leg is the ``output".}
\end{figure}

In the log geometric mirror construction, one can show that after fixing a point constraint, the tropical type above defines a moduli space of virtual dimension $0$, and the structure constants defining $R_{(X,D)}$ are the degrees of the relevant virtual fundamental class. In the non-archimedean mirror construction, a well behaved moduli space is not given directly from this tropical data, but only after extending the bounded leg to an extended leg. This extension is required to be straight in the affine structure put on $B$, given via an identification $B \cong M_\mathbb{R}$ which exists due to the choice of Zariski dense $n$ dimensional torus sitting inside $U$, see Lemma $2.2$ of \cite{archmirror}. Keel and Yu show these curves vary in connected families in the relevant ambient analytic moduli space of rational curves, and after picking a point in $p \in B \subset U^{an}$, there are a finite number of curves containing the point which have no automorphisms, and the structure constants are given by the degree of this 0-dimensional reduced scheme, or a simple count of curves. No virtual class is needed in this construction, as the relevant moduli space locally around contributing curves turn out to be smooth of the expected dimension, as established in Section $3$ of \cite{archmirror}. Using these respective definitions, the main results of \cite{archmirror} and \cite{int_mirror} establish that these structure constants give $R_{(X,D)}$ the structure of an $S_X$-algebra.

While the Gross-Siebert construction applies in greater generality, the Keel-Yu construction has a number of nice properties absent in the log geometric approach. Firstly, the structure constants are naive counts of curves, and in particular are non-negative integers, while there are examples of the Gross-Siebert mirrors in which there are negative structure constants. In addition, the non-archimedean mirror is defined as a family over $Spec\text{ }S_X$, and not only over a formal neighborhood of the closed point associated to the maximal monomial ideal. Finally, the non-archimedean mirror algebra has been shown to possess a non-degenerate trace given by taking the coefficient $\vartheta_0$. We review a standard Frobenius structure result which states that the structure constants are completely determined by the two and three point multilinear functions $a\otimes b \rightarrow tr(ab)$ and $a\otimes b \otimes c \rightarrow tr(abc)$.

\begin{proposition}\label{nondeg}
Consider a free $A$-module $V$ for a $\kk$-algebra $A$, two $A$ algebra structures $V_1$ and $V_2$ on $V$, and a basis $\{e_i\}$ of $V$ as a free $A$-module such that for both algebras $V_1$ and $V_2$, $1 = e_0 \in \{e_i\}$. Let $tr^k_{V_l}: V^{\otimes n} \rightarrow  A$ be the multilinear map such that $tr^k_{V_l}(v_1\otimes\ldots\otimes v_k)$ is the coefficient of $e_0 \in V$ in $v_1\cdot_{V_l} \ldots \cdot_{V_l} v_k$ expressed in the basis $\{e_i\}$. Suppose $tr_{V_1}$ is non-degenerate i.e., for any $a \in V\setminus 0$, there exists $b \in V$ such that $tr_{V_1}(a,b)\not= 0$. Then if $tr_{V_1}^i = tr_{V_2}^i$ for $i = 2,3$, then $a\cdot_{V_1} b = a \cdot_{V_2} b$ for all $a,b \in V$.
\end{proposition}

\begin{proof}
Since $tr_{V_1}$ is non-degenerate, an element $v \in V$ is uniquely determined by the linear map $tr_{V_1}(v,-): V \rightarrow A$. Indeed, $tr$ induces an $A$ linear map $tr(-,-):V \rightarrow V^{\vee}$, and the non-degeneracy condition implies this map is injective. Now note that for all $a,b,c \in V$, since we have assumed $2$ and $3$ input trace functions are equal, we have:
 
 \[tr_{V_1}(a\cdot_{V_1}b,c) = tr_{V_1}(a,b,c) = tr_{V_2}(a,b,c) = tr_{V_2}(a\cdot_{V_2}b,c) = tr_{V_1}(a\cdot_{V_2}b,c).\]
 Since this equality holds for all $c \in V$, $tr_{V_1}(a\cdot_{V_1} b,-) = tr_{V_1}(a\cdot_{V_2}b,-)$, and therefore $a\cdot_{V_1}b = a\cdot_{V_2}b$.

\end{proof}

Both algebras are equipped with such a trace by taking the $\vartheta_0$ coefficient of the product, and we note that by Theorem $1.2$ of \cite{archmirror}, the trace form on the Keel-Yu mirror algebra is non-degenerate. While this mirror algebra has coefficients in $\kk[NE(X)]$, it is easy to deduce from the polynomial case that the same non-degeneracy holds after extending scalars of the polynomial mirror algebra to $\kk[[NE(X)]]$.

Since the Gross-Siebert mirror algebra is a priori an algebra over the ring $\kk[[NE(X)]]$, if we show that these coincide for products of two and three theta functions, then the non-degeneracy of the trace of the Keel-Yu mirror will show that the structure constants defining both mirror algebras coincide. 

In order to study the trace form given by the Gross-Sibert mirror in the case of two and three inputs, we recall a result in section $9$ of \cite{int_mirror}, which allows us to express these values as certain ordinary log Gromov-Witten invariants:

Fix a general point $x \in U$, and $\beta$ the data of a curve class $\textbf{A}$, together with compatible contact orders $p_1,\ldots,p_k, 0$ at $k+1$ marked points. Let $\mathscr{M}(X,\beta)_x:= \mathscr{M}(X,\beta) \times_{X} x$ be the point constrained moduli space, which forces the marked point $x_{out}$ to map to $x$. Further define $\mathfrak{M}^{ev}(\mathcal{X},\beta)_x := \mathfrak{M}^{ev}(\mathcal{X},\beta)\times_X x$. We equip the natural morphism $\mathscr{M}(X,\beta)_x \rightarrow \mathfrak{M}^{ev}(\mathcal{X},\beta)_x$ with a perfect obstruction theory by pulling back the obstruction theory for $\mathscr{M}(X,\beta) \rightarrow \mathfrak{M}^{ev}(\mathcal{X},\beta)$. By \cite{int_mirror} Proposition $3.19$, the virtual dimension of this moduli space is $k-2$. Now define the following log Gromov Witten invariant:

\[N_{p_1,\ldots, p_k,0}^\textbf{A} = N_{\textbf{P}}^{\textbf{A}} = \int_{[\mathscr{M}(X,\beta)_x]^{vir}} \psi_{x_{out}}^{k-2}.\]
In the above expression, the $\psi$ integrand is the first Chern class of the cotangent line of $x_{out}$. 

\begin{theorem}[\cite{int_mirror}, Theorem 9.3]
For k = 2,3, the $\vartheta_0$ coefficient in $\vartheta_{p_1}\cdots \vartheta_{p_k}$ is $\sum_{\textbf{A} \in NE(X)} N_{p_1,\ldots,p_k,0}^\textbf{A}\ t^\textbf{A}$.
\end{theorem}

In the proof of our main result, we will find the $\psi$ class integrand ensures contributing curves have component containing $x_{out}$ with some fixed configuration of special points, similar to data that is fixed in the Keel-Yu mirror construction. Since we only need to show the trace forms correspond in the cases $k=2,3$ it suffices to show that the above log Gromov-Witten invariants equal the naive counts which compute the trace in the Keel-Yu mirror. Before proceeding to the proof of this statement, we need to introduce a class of tropical types which will play an important role throughout this paper.

\section{Tropical balancing condition and broken line types}
In this section, we will introduce various tropical types of log stable maps which appear when we consider log maps intersecting a zero stratum of $(X,D)$. These appear when we consider the point constraint $x \in X$ considered in Theorem $1.1$ degenerating to a zero stratum $s \in X$. Due to the degenerate nature of these curves, tropical geometry will offer a useful method of study of these families of curves. These tropical types will feature prominently in Steps $4$ and $5$ of the proof of Theorem \ref{mthm11}.

As introduced by Kontsevich and Soibelman in \cite{K3affine} and by Gross and Siebert in \cite{afftocpx}, a now longstanding method for understanding various mirror constructions has been through the use of scattering diagrams and in the latter paper broken lines. Given a PL manifold with singularities $B$, and a polyhedral decomposition $\mathscr{P}$ of $B$, then away from a discriminant locus, there exists an integral affine structure, allowing for a sensible notion of a tropical balancing condition. Broken lines then are certain piecewise linear maps $\mathbb{R}_{\ge 0} \rightarrow B$ which satisfy a generalization of the tropical balancing condition almost everywhere, except at finitely many points where the curve is allowed to bend in some manner prescribed by a scattering diagram on $(B,\mathscr{P})$, describing certain walls interacting with the discriminant locus of the integral affine structure. Tropical types analogous to broken lines will feature prominently in this paper. Before introducing the relevant types, we need a few definitions.

\begin{definition}
A realizable tropical type $\tau = (G, \pmb\sigma, \textbf u)$ is \emph{balanced} if for all vertices $v \in V(G)$ with $\pmb\sigma(v)$ a cone in $B_0 \subset \Sigma(X)$, and $t \in int(\tau)$, after postcomposing the universal map $\Gamma_\tau \rightarrow \Sigma(X)$ with the morphism $\Sigma(X) \rightarrow \Sigma(X_{good})$, and identifying resulting contact orders $\textbf{u}(e)$ for edges and legs $e$ containing $v$ with elements of the stalk of the local system $\Lambda_{h_t(v)}$ defined by the integral affine structure on $B_0$, then $\sum_{v \in e} \textbf u(e) = 0$. 
\end{definition}

\begin{remark}\label{rmk1}
The given definition of tropical balancing is different than the version given in \cite{scatt}, but coincides when our target of interest satisfies Assumption $1.1$ of \cite{scatt}. To see the utility of this slight generalization of tropical balancing, note that given $\sigma \in \Sigma(X)^{max} \setminus \mathscr{P}^{max}$ which contains a codimension $1$ face $\gamma \in \mathscr{P}^{[n-1]}$, then the image of $\sigma$ under the map of cone complexes $\Sigma(X) \rightarrow \Sigma(X_{good})$ is contained in $\gamma$. Thus, given a balanced tropical type $\tau$ and $v \in V(G_{\tau})$ with $\pmb\sigma(v) \in \mathscr{P}^{[n-1]}$ contained in two good zero dimensional strata $\gamma,\gamma'$, and $e \in E(G_{\tau})$ with $v \in e$ and $\pmb\sigma(e) = \gamma$, then tropical balancing forces there to be another edge or leg $e'$ containing $v$ with $\pmb\sigma(e')= \gamma'$. 

\end{remark}

The notion of balancing is important due to the following lemma:

\begin{lemma}[\cite{scatt} Lemma $2.1$]
Let $C^\circ/W \rightarrow X$ be a stable punctured log map, with $W = (Q \rightarrow Spec \text{ }\kk)$ a geometric log point. For $s \in int(Q^{\vee}_{\mathbb{R}})$, let $h_s: G \rightarrow \Sigma(X_{good})$ be the corresponding tropical map. If $v \in V(G)$ satisfies $h_t(v) \in B_0$, then $h_s$ satisfies the tropical balancing condition at $v$. 
\end{lemma}

If $(X,D)$ satisfies Assumption $1.1$ of \cite{scatt}, then this is Lemma $2.1$ of \cite{scatt}. If not, then we assume $U$ contains a dense torus instead, hence Lemma \ref{1strat} applies. Using this input, the same argument applied to the induced morphism $C^\circ/W \rightarrow X_{good}$ applies to show the resulting tropical type is balanced in the above looser sense. We now introduce the following tropical types:

\begin{definition}
A \emph{broken line type} $\tau = (G,\pmb\sigma,\textbf u )$ is a type of punctured log curve with target $X$ with $dim\text{ }X = n$ satisfying:

\begin{enumerate}

\item $G$ is a genus zero graph with $L(G) = \{L_{in},L_{out}\}$, with $\pmb\sigma(L_{out}) \in \mathscr{P}$ and $u_\tau := \textbf{u}(L_{out}) \not= 0$, $p_\tau := \textbf{u}(L_{in})\in \pmb\sigma(L_{in})\setminus\{0\}$, i.e. that the slope of the leg $L_{in}$ is non-negative with respect to all non-negative linear functions on the cone, and is not zero.
\item $\tau$ is realizable and balanced.
\item Let $h: \Gamma(G,l) \rightarrow \Sigma(X)$ be the corresponding universal family of tropical maps, and let $\tau_{out} \in \Gamma(G,l)$ the cone corresponding to $L_{out}$. Then dim $\tau = n-1$ and dim $h(\tau_{out}) = n$. 

\end{enumerate}
\end{definition}

We will occasionally also refer to a trivial broken line type, which is not an actual tropical type but refers only to a graph with one leg $L_{in} = L_{out}$ and no vertices, with the convention that $\pmb\sigma(L_{out}) = -\pmb\sigma(L_{in})$.

While we will not necessarily assume that $X$ has a canonical scattering diagram, the following tropical type will occasionally be useful to refer to:

\begin{definition}
 A \emph{wall type} $\omega = (G,\pmb\sigma, \textbf u)$ is a type of punctured log curve in a target $X$ satisfying:
\begin{enumerate}
\item $G$ is a genus $0$ graph with $L(G) = \{L_{out}\}$ with $\pmb\sigma(L_{out}) \in \mathscr{P}$, and $ u_\omega := \textbf{u}(L_{out}) \not= 0$. 
\item $\omega$ is realizable and balanced.
\item Let $h: \Gamma(G,l) \rightarrow \Sigma(X)$ be the corresponding universal family of tropical maps, and $\omega_{out}$ the cone corresponding to $L_{out}$. Then dim $\omega = n-2$, and dim $h(\omega_{out}) = n-1$. Furthermore, $h(\omega_{out}) \not\subset \partial B$. 
\end{enumerate}
\end{definition}

\begin{figure}
\centering
\begin{tikzpicture}
\fill[white!70!blue, path fading = north] (0,0)--(-3,0)--(-3,3)--(0,3)--cycle;
\fill[white!70!blue, path fading = north] (0,0)--(3,0)--(3,3)--(0,3)--cycle;
\fill[white!70!blue, path fading = south] (0,0)--(3,-3)--(3,0)--cycle;
\fill[white!70!blue, path fading = south] (0,0)--(3,-3)--(-3,-3)--(-3,0)--cycle;
\draw[black] (0,0)--(0,3);
\draw[black] (0,0)--(3,0);
\draw[black] (0,0)--(3,-3);
\draw[black] (0,0)--(-3,0);
\draw[dashed] (0,0)--(0,-3);
\draw[ball color = red] (0,1.5) circle (0.5mm);
\draw[ball color = red] (1.5,0) circle (0.5mm);
\draw[->,color = red] (0,1.5)--(-1.5,3);
\draw[-,color = red] (0,1.5)--(1.5,0);
\draw[ball color = green] (0,0) circle (0.5mm);
\draw[-,color = green] (0,0)--(1.5,0);
\draw[->,color = red] (1.5,0)--(3,-0.75);
\end{tikzpicture}
\caption{An example of a broken line type in the tropicalization of a blow up of $\mathbb{P}^1\times\mathbb{P}^1$ at a point contained in the interior of a boundary divisor. The $1$-dimensional cone contained in the lower half plane is identified with the dashed line to give the integral affine structure. The spine of the type is depicted in red, and a wall type produced by forgetting the spine is depicted in green, in this case corresponding to the exceptional divisor. Note that a tropical balancing condition holds at all vertices in the spine. Moreover, a broken line can only bend along the line spanned by this red ray.}
\end{figure}

After recalling the notion of a spine of a dual graph with legs, we note a tropical lemma which will restrict the behavior of the corresponding log curves, as well as allowing these types to be related to older notions of walls and broken lines.

\begin{definition}
Given a graph $(G,l)$, the spine $(S,l)$ of $(G,l)$ is the convex hull of the legs.
\end{definition}

\begin{lemma}[\cite{scatt} Lemma $2.5$]\label{GStlemma}
Fix a balanced type $\tau = (G,\pmb\sigma,\textbf{u})$ of a genus zero tropical map to $\Sigma(X)$ with a distinguished leg $L_{out} \in L(G)$ and $\textbf{u}(L_{out}) \not= 0$. Let $h: \Gamma(G,l) \rightarrow \Sigma(X)$ be the corresponding universal family of tropical maps, defined over the cone $\tau$. 
\begin{enumerate}
\item Suppose $G$ has only one leg, $L_{out}$, and let $\tau_{out} \in \Gamma$ be the corresponding cone. Suppose further that $\pmb\sigma(L_{out}) \in \mathscr{P}$. Then dim $h(\tau_{out}) \le n-1 $.
\item Suppose $G$ has precisely two legs, $L_{in}$ and $L_{out}$, with $\pmb\sigma(L_{in})$, $\pmb\sigma(L_{out}) \in \mathscr{P}$. Suppose further that dim $\tau = n-1$ and dim $h(\tau_{out}) = n$. Then with $S$ the spine of $G$, and $t\in int(\tau)$, $h_t(S) \subset B \subset |\Sigma(X)|$, and $h_t(S)$ only intersects codimension zero and one cones of $\mathscr{P}$, except possibly at non-vertex endpoints of $L_{in}$ and $L_{out}$. Furthermore, dim $h_{\tau}(\tau_v) = n-1$ for every vertex of the spine $v \in V(S)$. 

\end{enumerate}
\end{lemma}

Again, while we do not assume in general that $(X,D)$ satisfies Assumption $1.1$ of \cite{scatt}, we note that using the conclusion of Remark \ref{rmk1} as a replacement for Assumption $1.1(2)$ of \cite{scatt} is sufficient to conclude the above lemma for all cases in our context. We note the following corollary of Lemma \ref{GStlemma}:

\begin{corollary}\label{mincor}
For a broken line type $\tau$, and $v \in V(S)$ a vertex contained in the spine, the morphism $ev_{v}: \tau \rightarrow \pmb\sigma(v)$ is injective, i.e. if two punctured tropical maps $\Gamma,\Gamma' \rightarrow \Sigma(X)$ of type $\tau$ map $v$ to the same point in $\pmb\sigma(v)$, then the two tropical maps are isomorphic.
\end{corollary}

\begin{proof}
By Lemma \ref{GStlemma}(2), we have $im(ev_{v}) = h_{\tau}(\tau_{v}) = n-1$ Since $dim\text{ }\tau = n-1$, and $ev_{v}$ is a map of cones, we must have $ev_{v}$ is injective.
\end{proof}

For future purposes, we define $k_\tau:= |coker(ev_{v_{out}}: \tau_{out,\NN}^{gp} \rightarrow \pmb\sigma(L_{out})^{gp}_{\NN})|$, which is finite since $dim\text{ }h(\tau_{out}) = n = dim\text{ }\pmb\sigma(L_{out})$, and the map is injective.

A classical broken line comes from a broken line type by taking the spine of a family of tropical curves of type $\tau$. Moreover, components of the complement of the family of spines correspond to walls which allow for the broken line to bend in the integral affine structure on $B_0$. For more details on this correspondence, see Section $4$ of \cite{scatt}. 

We will primarily be using the following corollary of Lemma \ref{GStlemma}:

\begin{corollary}\label{tlemma1}
Let $\tau = (G,\pmb\sigma,\textbf{u})$ be a tropical type of a stable log map $f: C \rightarrow X$, with universal family of tropical curves $h: \Gamma(G,l) \rightarrow \Sigma(X)$, marked by a broken line tropical type $\tau' = (G',\pmb\sigma',\textbf{u})$. Let $S,S'$ be the spines of $G,G'$ respectively. Then $dim\text{ }h(\tau_v) = n-1$ for all $v \in V(S)$, $\textbf{u}(e) \notin h(\tau_v)^{gp}$ for all $e \in E(S)$ containing a vertex $v \in V(S)$, and if $t \in \tau$ with $h_t(v) = 0$ for some $v \in V(S)$, then $h_t(v') = 0$ for all $v' \in V(S)$.

\end{corollary}

\begin{proof}

For $t \in \tau'$ and $v' \in V(S')$, by Corollary \ref{mincor}, $h_t(v')$ uniquely determines the tropical map $h_t: \Gamma_t' \rightarrow \Sigma(X)$, i.e., $ev_{v'}:\tau' \rightarrow \pmb\sigma(v')$ is injective. By the relationship between basic monoids described in Remark $3.7$ of \cite{punc}, note that since we assumed $\tau$ is marked by the realizable tropical type $\tau'$, we may identify $\tau'$ with a face of $\tau$. Let $c: \Gamma(G',l) \rightarrow \Gamma(G,l)$ be the morphism of universal families associated with the face inclusion $\tau' \subset \tau$, and a corresponding graph contraction $\phi: G \rightarrow G'$. Note that $c$ maps the cones corresponding to vertices and edges in the spine $S'$ of $G'$ to the cones of the spine $S$ of $G$. In particular, the contraction map $\phi: G \rightarrow G'$ restricts to a morphism $\phi|_{S}: S \rightarrow S'$. We label the corresponding subcomplexes of $\Gamma(G',l')$ and $\Gamma(G,l)$ as $\Gamma(S',l)$ and $\Gamma(S,l)$ respectively. Let $v \in V(S)$, and note by Lemma \ref{GStlemma} that $\pmb\sigma(\phi(v)) \in \mathscr{P}$ and dim $h'(\tau'_{\phi(v)}) = n-1$. Hence dim $h(\tau_v) \ge n-1$.

If $\pmb\sigma(v) \in \mathscr{P}^{[n-1]}$, then clearly $dim\text{ }h(\tau_v) = n-1$. If $\pmb\sigma(v) \in \mathscr{P}^{max}$, then the generic point of the corresponding component $C_v$ of $C$ maps to a zero stratum, so $f$ contracts $C_v$. Thus, by stability of $f$, $v$ must be at least $3$ valent. Letting $e'$ be an edge containing $v$ but not contained in $E(S)$, consider the tropical type resulting from cutting at $e'$ and taking the connected component not containing $v$. The resulting tropical type satisfies the conditions of Lemma \ref{GStlemma} (1), hence $h(\tau_v)$ is contained in a codimension $1$ subcone of $\pmb\sigma(v)$. Therefore, dim $h(\tau_v) = n-1$ for all $v \in V(S)$ with $\pmb\sigma(v) \in \mathscr{P}$. In particular, the inclusion of cones $h(\tau'_{\phi(v)}) \subset h(\tau_v)$ induces an isomorphism of real vector spaces $h(\tau'_{\phi(v)})^{gp} \rightarrow h(\tau_v)^{gp}$.

Let $e \in E(S)\cup L(S)$ be an edge or leg in the spine, with associated contact order $\textbf{u}(e) \in \Lambda_{\pmb\sigma(e)}$. The image of $\tau_e$ in $\Gamma(G',l)$ lies in the subcomplex $\Gamma(S',l)$. If $e$ is not contracted by $\phi$, then by Lemma \ref{GStlemma}, $\pmb\sigma(e) = \pmb\sigma(\phi(e)) \in \mathscr{P}^{max}$ and the tangent vector $\textbf{u}(e)$ is not contained in $h(\tau'_{\phi(v)})^{gp}$ for any vertex $v \in E$, hence is not contained in $h(\tau_v)^{gp}$. Now suppose $v \in V(S)$ is contained in a non-contracted edge or leg $e \in E(S)\cup L(S)$. Then $\pmb\sigma(v) \in \mathscr{P}$ and $dim\text{ }h(\tau_v) = n-1$. Thus $h_t(v) \in B_0$ for $t \in int(\tau)$. By the tropical balancing condition defined in $B_0$, there must exist an edge $e' \in G$ containing $v$ such that $\pmb\sigma(e') \in \mathscr{P}^{max}$ and $\textbf{u}(e')$ is not contained in $h(\tau_v)^{gp}$. 

Now suppose additionally that $e' \notin E(S)\cup L(S)$. Then $e'$ cannot be a leg of $G$ so there must be a vertex $v' \not= v$ contained in $e'$. Since $dim\text{ }h(\tau_v) = n-1$, $dim\text{ }h(\tau_{e'}) = n$, and $dim\text{ }h(\tau_{v'}) \ge n-1$. Since $\pmb\sigma(e') \in \mathscr{P}^{max}$, we must have $h_t(v') \in B_0$ for $t \in int(\tau)$. By the same argument constraining the dimension of $h(\tau_v)$ for $v \in V(S)$, we have $dim\text{ }h(\tau_{v'}) = n-1$. Thus, $\textbf{u}(e') \notin h(\tau_{v'})^{gp}$, and by the tropical balancing condition, there must exist another edge $e''$ containing $v'$ such that $\pmb\sigma(e'') \in \mathscr{P}^{max}$, with $\textbf{u}(e'') \notin h(\tau_{v'})^{gp}$. By repeating this construction, we create an infinite chain of edges and vertices, a contradiction since the dual graph must be genus $0$. Thus, we must have $e' \in E(S) \cup L(S)$. Since $S$ is a linear subgraph with two legs, all edges $e \in E(S)$ must satisfy $\pmb\sigma(e) \in \mathscr{P}^{max}$ and $\textbf{u}(e) \notin h(\tau_v)^{gp}$ for all $v \in e$. In particular, $\pmb\sigma(v) \in \mathscr{P}$ for all $v \in V(S)$ as well.

Finally, let $v,v' \in V(S)$ be vertices contained in an edge $e \in E(S)$, with contact order $\textbf{u}(e)$ after picking an orientation going from $v$ to $v'$. Then for $t \in \tau$, $h_t(v')$ is determined by $h_t(v)$, given by the intersection of the straight ray starting at $v$ with outgoing vector $\textbf{u}(e)$, with the cone $h(\tau_{v'})$. By the transversality of $\textbf{u}(e)$ with respect to the cone $h(\tau_{v'})$, $h_t(e)$ uniquely determines $h_t(v')$. In particular, if $h_t(v) = 0$, then $h_t(v') = 0$. Proceeding by an induction on distance from $v$, we find that $h_t(v) = 0$ implies $h_t(v') = 0$ for all $v' \in V(S)$. 
\end{proof}

In Section $6$ of \cite{scatt}, Gross and Siebert show that when $(X,D)$ satisfies Assumption $1.1$ of loc cit.  the structure constants of the intrinsic mirror are weighted counts of pairs of broken lines meeting at a point, with weights involving the virtual counts of log maps marked by associated with two broken line types. Roughly speaking, the relationship between the two log geometric mirror constructions is that curves contributing to the structure constants in the intrinsic mirror symmetry construction degenerate to pairs of logarithmic curves glued along a zero stratum $s \in X$, which tropicalizes to a pair of broken lines intersecting at some generic point in the cone $\sigma \in \Sigma(X)$ associated with the stratum $s \in X$. This is described in Section $6$ of \cite{scatt}, and uses Lemma $7.8$ of \cite{int_mirror} in a key way to construct this degeneration. We will make use of a similar argument in the proof of Theorem \ref{mthm11}.

\section{Comparing the Trace Forms}

With most of the necessary background established, we turn our attention to the setting of our main theorem. As discussed at the end of Section $3$, both mirror algebras are equipped with the same identity element given by a theta function $\vartheta_0$, and both have a natural trace form by taking the $\vartheta_0$ coefficient of a product $\vartheta_{p_1}\cdots \vartheta_{p_k}$. Keel and Yu show that the trace form of their non-archimedean mirror recovers a certain naive count of curves. This naive count of curves, which we will now recall, is well defined without the torus assumption

Fix a log Calabi-Yau pair $(X,D)$ satisfying the conditions of Theorem \ref{mthm11}, with essential skeleton $B$. Let $\textbf{P} := (p_1,\ldots,p_k,0) \in B(\ZZ)^{k+1}$ with $p_i = m_id_i \in \sigma_{D_i} \in \mathscr{P}^{[1]}$ with $d_i \in \mathscr{P}^{[1]}$ primitive integral points along the $1$-dimensional cone associated with $D_i$, and  $m_i \in \ZZ_{\ge 0}$. Further, let $\mathscr{M}^{tr}(U,\textbf{A},\textbf{p})$ be the non-proper moduli space of genus $0$ stable curves in $X$ with $k+1$ marked points of class $\textbf{A}$ such that $x_i$ maps to the interior of the divisor $D_i$, and $f^{-1}(D) = \sum_i m_ix_i$ and $f(x_i) \in D_i \setminus \{\overline{D\setminus D_i}\}$. We note that the property $p_i \in \sigma_{D_i}$ is always satisfied after an appropriate toric blowup, and we define $\mathscr{M}^{tr}(U,\textbf{A},\textbf{p})$ to be the corresponding moduli space associated with rational curves in the blowup. We call integral points of $\sigma_{D_i} \in \mathscr{P}^{[1]}$ divisorial. Following this definition, Keel and Yu define in Notation $3.3$ of \cite{archmirror} a moduli space $\mathscr{M}^{sm}(U,\textbf{A},\textbf{p}) \subset \mathscr{M}^{tr}(U,\textbf{A},\textbf{p})$ as an open subvariety satisfying various open properties, but important for us is that the domain is a stable curve and $\tilde{f}^*(T_{\tilde{X}}(-log$ $\tilde{D}))$ is the trivial vector bundle for all curves $(\tilde{f},\tilde{C}) \in \mathscr{M}^{sm}(U,\textbf{A},\textbf{p})$. By Lemma $3.6$ of \cite{archmirror}, this open locus is smooth of dimension $k-2 + \text{dim }X$. The next proposition gives a point constraint and curve modulus criterion for guaranteeing a geometric point $Spec\text{ }\kk \rightarrow \mathscr{M}^{tr}(U,\textbf{A},\textbf{p})$ factors through $Spec\text{ }\kk\rightarrow \mathscr{M}^{sm}(U,\textbf{A},\textbf{p})$.

\begin{theorem}[\cite{archmirror}, Theorem 3.13]\label{smmod}
Let $\beta = (\textbf{p},\textbf{A})$, where $\textbf{A} \in NE(X)$, and $\textbf{p} = (p_1,\ldots,p_k,0)$ are contact orders contained in the Kontsevich-Soibelman skeleton of $D$, assumed to be divisorial. We denote the marked point with contact order $0$ by $x_{out}$. Then there is an open subset $V \subset U \times \mathscr{M}_{0,k+1}$ such that the following hold:
\begin{enumerate}
\item The preimage of $V$ by the map $\Phi_{x_{out}}: \mathscr{M}^{tr}(U,\textbf{A},\textbf{p}) \rightarrow U\times \overline{\mathscr{M}_{0,k+1}}$ is contained in $\mathscr{M}^{sm}(U,\textbf{A},\textbf{p})$. 
\item The map $\Phi_{x_{out}}$ is representable and finite \'etale over $V$. 

\end{enumerate}
\end{theorem}

Using this morphism, Keel and Yu define a trace $\eta: R^k \rightarrow S_X$ defined by:
\[ \eta(p_1,\ldots,p_k) := \sum_{\textbf{A} \in NE(X)} \eta(p_1,\ldots,p_k,\textbf{A})t^\textbf{A} \in R\]
 where $\eta(p_1,\ldots,p_k,\textbf{A})$ is the degree of $\Phi_{x_{out}}$. For notational purposes, we will consider the vector $\textbf{P}=(p_1,\ldots,p_k)$ of contact orders, and denote the above count by $\eta(\textbf{P},\textbf{A})$. We note that the proof of Theorem \ref{smmod} does not assume $U$ contains a Zariski dense torus. Assuming the dense torus exists however, it is shown in section $15$ of \cite{archmirror} that this trace is the $\vartheta_0$ component of the product of the non-archimedean theta function $\vartheta_{p_1}\cdots \vartheta_{p_k}$. 

The above theorem also says something about the log GW invariant we wish to compute. First, note that the data $\textbf{A}$ and $\textbf{p} = (p_1,\ldots,p_k,0)$ defines a tropical type $\beta$.  We will show that all curves in the first moduli space admit unique lifts to log curves of type $\beta$, hence to show that we have an embedding $\mathscr{M}^{tr}(U,\textbf{A},\textbf{p})\subset \mathscr{M}(X,\beta)$. For our purposes, we will only need the case in which the contact orders are associated with components of the boundary.
\begin{lemma}\label{tlift}
Consider a family of stable maps $(\underline{C}/\underline{S},\underline{f})$ defined by a morphism $\underline{S} \rightarrow \mathscr{M}^{tr}(U,\textbf{A},\textbf{p})$, with $\beta$ a tropical class with divisorial contact orders along legs. Then there exists a unique log structures on $C$ and $f$ which lift $(C/S,f)$. In particular, we can identify $\mathscr{M}^{tr}(U,\textbf{A},\textbf{p})$ with an open locus inside $\mathscr{M}^{log}(\tilde{X},\beta)$ consisting of log curves with decorated tropical type given by $\beta$.
\end{lemma}

\begin{proof}
First, we note that the underlying family of prestable curve $C \rightarrow S$ is given by a morphism $S \rightarrow \textbf{M}$ to the stack of prestable curves over $Spec\text{ }\kk$, which by \cite{LogGW} Appendix $A.3$ is an algebraic stack admitting a log structure given by the basic log structure. Using this, we can equip $C$ and $S$ with log structures $\mathcal{M}^{pre}_C$ and $\mathcal{M}_S$ such that $C \rightarrow S$ is log smooth. Since $f: C \rightarrow X$ intersects the boundary at finitely many points, the pullback log structure $ f^*\mathcal{M}_X$ has a uniquely determined map to the divisorial log structure on the section divisor of $C$ which intersects $D \subset X$, which we denote by $\mathcal{M}_C^{div}$. We now equip $C$ with the log structure $\mathcal{M}_C = \mathcal{M}^{pre}_C \oplus_{\mathcal{O}_C^\times} \mathcal{M}_C^{div}$. With these log structures, $C \rightarrow S$ is log smooth, and we have a log lift $C \rightarrow X$ of the underlying morphism of schemes. Uniqueness of a log lift follows from an easy case of \cite{decomp} Theorem $4.13$.

\end{proof}

Thus, all curves contributing to the Keel-Yu trace defined above admit unique lifts to log curves. Around any such map $[C,f]$, by the triviality of $f^*(T_{\tilde{X}}^{log})$ and Lemma $3.6$ of \cite{archmirror}, the map $\mathscr{M}(\tilde{X},\beta) \rightarrow \mathscr{M}_{0,k+1} \times X$ is locally \'etale around a curve contributing to the Keel-Yu trace, and contributes once to log Gromov-Witten invariant $N_{p_1\ldots,p_k,0}^{\textbf{A}}$. The central question which needs to be resolved to compare the Keel-Yu trace with the log Gromov-Witten invariant $N_{p_1,\ldots,p_k,0}^\textbf{A}$ is that all curves in the the support of a generic representative of the above degree $0$ cycle are one of these curves, in particular do not have components contained in the boundary.

We first recall two definitions and two lemmas from \cite{int_mirror} necessary for our main theorem. First, in logarithmic geometry, we are in general interested in working with log flat families over a base. In order to give a notion of dimension in which log flat families have constant fiber dimension, we define log fiber dimension, defined originally in \cite{logdim}:

\begin{definition}\label{fibdim}
Let $f: X \rightarrow Y$ be a log flat morphism, with underlying scheme map locally of finite presentation and $x \in X$. Set $y = f(x)$, $X_y = X \times_{Spec\text{ }\kappa(y)} Spec\text{ }\kappa(x)$, where $\kappa(x)$ and $\kappa(y)$ are the residue fields of $X$ at $x$ and $Y$ at $y$ respectively. Now define the log fiber dimension of $f$:

\[
dim_x^{log}f^{-1}(f(x)) := 
\]
\begin{equation}
dim \text{ } \mathcal{O}_{X_y,x}/(\alpha(\mathcal{M}_{X,x} \setminus \mathcal{O}_{X,x}^{\times})) + tr.deg\text{ }\kappa(x)/\kappa(y) + rank\text{ }\overline{\mathcal{M}}_{X,x}^{gp} - rank\text{ }\overline{\mathcal{M}}_{Y,y}^{gp}.
\end{equation}
\end{definition}

In the above definition, note that the first term is the dimension of the local ring of $x$ in the stratum of $X_y$ containing $x$. In particular, if $x$ is a generic point of $X_y$, then $dim \text{ } \mathcal{O}_{X_y,x}/(\alpha(\mathcal{M}_{X,x} \setminus \mathcal{O}_{X,x}^{\times})) = dim\text{ }\mathcal{O}_{X_y,x}$.   The above definition may be extended to algebraic log stacks by taking appropriate smooth charts.

Now recall the definition of the ordinary fiber dimension:

\begin{equation}
dim_xf^{-1}(f(x)) = dim\text{ }\mathcal{O}_{X_y,x} + tr.deg  \text{ }\kappa(x)/\kappa(y).
\end{equation}

In general, only the former is locally constant for a log flat morphism. However, if $f$ is $\mathbb{Q}$-integral, as defined in \cite{Og} I.4.7.4, then by \cite{int_mirror} Proposition $A.7(2)$, the two are equal in an \'etale neighborhood of $x$. Moreover, the following definition and lemma indicate how we may produce such morphisms, starting from a morphism $X \rightarrow Y$ which is only log smooth:

\begin{definition}[\cite{int_mirror} Definition 2.6]
Let $f: X \rightarrow Y$ be a morphism of fs log stacks. A morphisms $g: W \rightarrow Y$ is \emph{transverse} to $f$ if $W\times^{fs}_Y X \rightarrow W$ is integral. 
\end{definition}

\begin{lemma}[\cite{int_mirror} Theorem 2.9]\label{tran}
Let $f: X \rightarrow Y$, $g: W \rightarrow Y$ be log morphisms between fs log stacks. Suppose further that $W$ is free, that is, the stalks $\overline{\mathcal{M}}_{W,w}$ are free monoids for all $w \in W$. Then $g$ is transverse to $f$ if for any geometric point $x \in X$, $y \in Y$ and $w \in W$ with $f(x) = y = g(w)$, and for any face $F$ of $\sigma_x$ we have that $\Sigma(g)^{-1}(\Sigma(f)(F)) \subset \sigma_w$ is a face of $\sigma_w$.
\end{lemma}

In our context, the log smooth morphism we consider is $ev_{x_{out}}: \mathfrak{M}^{ev}(\mathcal{X},\beta) \rightarrow X$. By the above lemma, after making an appropriate base change, we may guarantee the morphism is flat, and the log and ordinary fiber dimensions coincide, which will be key in constructing a sensible degeneration of the moduli stack $\mathfrak{M}^{ev}(\mathcal{X},\beta)_x$ studied in Step $4$ of the proof of Theorem \ref{mthm11}:

\begin{lemma}[\cite{int_mirror} Lemma 7.8]\label{flat}
Let $W \rightarrow X$ be a morphism of log schemes, which is transverse to $\mathfrak{M}^{ev}(\mathcal{X},\beta) \rightarrow X$. Letting $ev_{x_{out}}: \mathfrak{M}^{ev}(\mathcal{X},\beta) \rightarrow X$ be the morphism of log stacks which evaluates a log curve at the contact order $0$ marked point $x_{out}$, then the projection $\mathfrak{M}^{ev}(\mathcal{X},\beta)_W := W \times^{fs}_X \mathfrak{M}^{ev}(\mathcal{X},\beta) \rightarrow W $ is flat with equal log and standard fiber dimension.
\end{lemma}

\begin{proof}
Since $ev_{x_{out}}: \mathfrak{M}^{ev}(\mathcal{X},\beta) \rightarrow X$ is log smooth, hence is log flat, by definition of transversality and log smooth morphisms being preserved under fine and saturated pullback, the projection $W \times^{fs}_X \mathfrak{M}^{ev}(\mathcal{X},\beta) \rightarrow W$ is integral and log flat. By \cite{int_mirror} Proposition $2.3(2)$, this projection is flat. The equality of log and standard fiber dimension follows from Proposition $A.7(2)$ of \cite{int_mirror}.
\end{proof}
Finally, we note that the definition of $\eta(\textbf{P},\textbf{A})$ is a naive count of curves on a blowup $\tilde{X}$ of $X$ along strata. We note the following birational invariance property for the log \'etale invariants of interest:

\begin{lemma}\label{birinv}
Let $p: \tilde{X} \rightarrow X$ be a log \'etale modification such, and after identifying $\Sigma(X)(\NN)$ with $\Sigma(\tilde{X})(\NN)$, let $\textbf{P}$ denote contact orders on $\tilde{X}$ as well. Then for any $\textbf{A} \in NE(X)$, there is a unique class $\textbf{A}' \in NE(\tilde{X})$ such that the following equation holds:

\[N_{\textbf{P}}^{\textbf{A}} = \sum_{p_*(\textbf{A}')=\textbf{A}} N_{\textbf{P}}^{\textbf{A}'}.\]

\end{lemma}

\begin{proof}
This essentially follows from the main theorem of \cite{bir_GW}, or via cosmetic changes to the proof of a special case of Corollary $1.6$ of \cite{pbirinv}. We provide a proof for completeness. 

Note that since $\mathscr{M}(X,\textbf{P})_x = \mathscr{M}(X,\textbf{P})\times_x X$ and the obstruction theory for the morphism $\mathscr{M}(X,\textbf{P})_x \rightarrow \mathfrak{M}^{ev}(\mathcal{X},\textbf{P})_x$ is defined by pulling back the relative obstruction theory for the morphism $\mathscr{M}(X,\textbf{P}) \rightarrow \mathfrak{M}^{ev}(\mathcal{X},\textbf{P})$, the pushforward of the virtual fundamental class $[\mathscr{M}(X,\textbf{P})_X]^{vir}$ to $A_*(\mathscr{M}(X,\textbf{P}))$ is the virtual pullback of the Chow class of the closed substack $\mathfrak{M}^{ev}(\mathcal{X},\textbf{P})\times_{X} x \subset \mathfrak{M}^{ev}(\mathcal{X},\textbf{P})$.
This Chow cycle in turn is given by $[\mathfrak{M}^{ev}(\mathcal{X},\textbf{P})]\cap ev^*_{x_{out}}([pt])$. We thus have the alternate definition for the invariants of interest:
\[N_{\textbf{P}}^{\textbf{A}} = \int_{[\mathscr{M}(X,\textbf{P},\textbf{A})]^{vir}} \psi_{x_{out}}^{k-2}\cup ev_{x_{out}}^*([pt]).\]
A similar expression holds for $N_{\textbf{P}}^{\textbf{A}'}$ for $\textbf{A}' \in NE(\tilde{X})$. Abusing notation slightly in reference to evaluation maps, we also have $p^*ev_{x_{out}}^*([pt]) = ev_{x_{out}}^*([pt])$. It now follows from Corollary $1.3.1$ of \cite{bir_GW} that $N_{\textbf{P}}^\textbf{A} = \sum_{p_*(\textbf{A}') = \textbf{A}} N_{\textbf{P}}^{\textbf{A}'}$. 

\end{proof}

We note that the conditions assumed in Theorem \ref{mthm11} also hold for a log \'etale modification $(\tilde{X},\tilde{D})$. Thus, if we know Theorem \ref{mthm11} holds for $\tilde{X}$ with discrete data $\textbf{P}$ and a curve class $\textbf{A}$, then after identifying the integral points of the essential skeletons of $\tilde{X}$ and $X$, Theorem \ref{mthm11} holds for $X$ with discrete data $\textbf{P}$ and curve class $\pi_*(\textbf{A})$ by definition of $\eta(\textbf{P},\textbf{A})$. Since there exists a blowup $\tilde{X} \rightarrow X$ such that the resulting contact orders $p_i \in \Sigma(\tilde{X})$ are divisorial, we may assume the contact orders $p_i \in \Sigma(X)(\NN)$ are divisorial.

With the necessary preliminary definitions and lemmas established, we now proceed to the proof of our main theorem:

\begin{proof}[Proof of Theorem \ref{mthm11}]

The proof will proceed by carefully picking a degeneration of our point constraint so that it induces a well behaved degeneration of the relevant point constrained moduli spaces, carried out in step $3$. After constructing this degeneration, in steps $4$ and $5$ with input from step $2$ to provide a modular interpretation of the insertion $\psi_{x_{out}}^{k-2}$ appearing in the definition of $N_{p_1,\ldots,p_k,0}^{\textbf{A}}$, we will study constraints on the tropical behavior of the special fiber, and use this to restrict the behavior of the general fiber. These constraints, combined with those coming from Step $1$ due to the semipositivity assumption i.e. the existence of a nef divisor $\sum_i a_iD_i$ with $a_i > 0$, are enough to constrain all possibly contributing curves with a general point condition to the smooth locus studied by Keel and Yu, from which the conclusion will follow. 

\textbf{Step 1: No unmarked rational tree in boundary}

We first show that certain types of subtrees of rational curves cannot appear in curves $f: C \rightarrow X$ associated to a point of $\mathscr{M}(X,\beta)$. Specifically, letting $\tau = (G_\tau,\pmb\sigma,\textbf{u})$ be the tropical type of $f:C\rightarrow X$, we claim that if $v \in V(G_\tau)$ is connected to a vertex $v_0$ by an edge $e$ such that $\pmb\sigma(v_0) = 0 \in \Sigma(X)$, and $\pmb\sigma(v)\not= 0$, then the closure of the component of $G_\tau\setminus \{e\}$ containing $v$ must contain a leg. We highlight an example of the disallowed behavior in the figure below.

\begin{figure}[h]
\centering
\begin{tikzpicture}
\fill[white!70!blue, path fading = north] (0,0)--(-3,0)--(-3,3)--(0,3)--cycle;
\fill[white!70!blue, path fading = north] (0,0)--(3,0)--(3,3)--(0,3)--cycle;
\fill[white!70!blue, path fading = south] (0,0)--(3,-3)--(3,0)--cycle;
\fill[white!70!blue, path fading = south] (0,0)--(3,-3)--(-3,-3)--(-3,0)--cycle;
\draw[black] (0,0)--(0,3);
\draw[black] (0,0)--(3,0);
\draw[black] (0,0)--(-3,0);
\draw[black] (0,0)--(0,-3);
\draw[ball color = red] (0,0) circle (0.5mm);
\draw[->, color = red] (0,0)--(0,3);
\draw[->, color = red] (0,0)--(-3,-3);
\draw[-,color = red] (0,0)--(1,0);
\draw[ball color = red] (1,0) circle (0.5mm);
\draw[-,color = red] (1,0)--(2,1);
\draw[ball color = red] (2,1) circle (0.5mm);
\draw[-,color = red] (1,0)--(2,-1);
\draw[ball color = red] (2,-1) circle (0.5mm);

\end{tikzpicture}
\caption{An example of behavior disallowed by Step $1$. The tail without legs is disallowed by the semipositivity assumption.}
\end{figure}

 The key input to this constraint is the semipositive assumption placed on the boundary, together with the logarithmic balancing condition of stable log maps.  First, let $D' = \sum_{i} a_iD_i$ with $a_i >0$ be the nef divisor supported on $D$, and suppose $v$ is not contained in the spine of $G_\tau$. Letting $q_e \in C$ be the node associated to the edge $e$, consider the closure of the connected component of the $C \setminus \{q_e\}$ which contains the generic point of the $C_v \subset C$. By restricting the log structure on $C$ to this union of components, we produce a punctured log curve $C' \rightarrow X$. Since this curve does not contain a marked point inherited from $C$, $C'$ has only one punctured point corresponding to the former node $q_e$. Since $\pmb\sigma(v_0) = 0$ but $\pmb\sigma(v) \not= 0$, by orienting the edge to point towards $v_0$, we must have $\textbf{u}(e)(D_i)\le 0$ for all $i$, and $\textbf{u}(D_j)<0$ for some $j$.  Thus, by applying the logarithmic balancing condition of \cite{punc} Proposition $2.27$, we find:
\[C' \cdot D' = \sum_i a_i \textbf{u}(e)(D_i) < 0.\]
This contradicts the assumption that $D'$ is nef.

\textbf{Step 2: $\psi$ class as fixing modulus of domain} 

Let $Forget: \mathscr{M}(X,\beta)_x \rightarrow \overline{\mathscr{M}}_{0,k+1}$ be the morphism corresponding to the stabilization of the universal family, and denote by $\psi_{x_{out}}$ and $\overline{\psi}_{x_{out}}$ the first Chern class of cotangent line at $x_{out}$ of $\mathscr{M}(X,\beta)_x$ and $\overline{\mathscr{M}}_{0,k+1}$  respectively. We wish to show that the equality of insertions:
\[Forget^*\overline{\psi}_{x_{out}} = \psi_{x_{out}} \in A^1(\mathscr{M}(X,\beta))_x\]
Standard results regarding $\psi$ classes imply the divisor $\psi_{x_{out}} - Forget^*(\overline{\psi}_{x_{out}})$ on $\mathscr{M}(X,\beta)_x$ is supported on the locus consisting of curves where the component containing $x_{out}$ is destabilized by the forgetful map, remembering only the underlying family of curves of the source. We wish to show that this class is zero. Let $[C,(x_i),x_{out},f] \in \mathscr{M}(X,\beta)_x$ be a geometric point, and note that the point constraint enforces $f(x_{out}) = x \in \pi_X(V) \subset U$. As part of the construction of $V$ in Section $3$ of \cite{archmirror}, specifically in Lemma $3.11$ which itself uses Lemma $5.11$ of \cite{RCoqp}, by the generic choice of $x$, for any rational curve $C''\subset X$ containing $x\in X$ which makes contact with the boundary at most once, the degree of $C''$ with respect to a choice of ample line bundle $L$ is strictly greater than $L\cdot \textbf{A}$. Since the component $C' \subset C$ containing $f(x_{out}) = x$ must have degree bounded above by $L\cdot \textbf{A}$, $C'$ must make contact with the boundary at at least two distinct points. Since the restriction of the log structure on $C$ to $C'$ gives a punctured log map which generically maps to $U$, all points of contact of $C'$ with the boundary must either be a marked point or a node. Picking two distinct points of $C'$ making contact with the boundary, if both are marked points, then the underlying component is stable since then $C'$ will have at least $3$ marked points. On the other hand, if one of the points is a node $q$, then by taking the complement of the node and the closure of the component of $C$ which does not contain $C' \setminus \{q\}$, by step $1$, this curve must contain a marked point. Hence, the image of $q$ under the stabilization map is still a special point. Thus, the component $C'$ is not contracted after stabilizing the underlying curve. Thus, $\psi_{x_{out}} - Forget^*(\overline{\psi}_{x_{out}}) = 0$, and we may replace $\psi_{x_{out}}$ in the definition of $N_{p_1,\ldots,p_k,0}^\textbf{A}$ with $Forget^*(\overline{\psi}_{x_{out}})$.

\textbf{Step 3: Degenerate point constraint}

We will study the log Gromov-Witten invariant of interest by degenerating the point constraint, similar to arguments appearing in \cite{scatt} Section $6$. Slightly more precisely, we will consider a family of point constraints which specialize to a zero stratum of $D$. We will see that with some care, this has the effect of degenerating our point constrained moduli spaces, giving a refinement of invariant of interest. 

Let $s \in X$ be a zero dimensional good stratum, with $\overline{\mathcal{M}}_{X,s} = P_s$ and corresponding cone $P^{\vee}_{s,\mathbb{R}} = \sigma_s \in \Sigma(X)$. Since $(X,D)$ is log smooth, there exists an \'etale map $i: O \rightarrow X$ which $s: Spec \text{ }\kk \rightarrow X$ factors through, and $j: O \rightarrow Spec\text{ }\kk[P_s]$ a strict \'etale map with $s \in O$ mapping to $0 \in Spec\text{ }\kk[P]$. Note we may let $O$ be connected. Now consider an element $u \in int(\sigma_s)\cap \sigma_{s,\NN}$, and note this element, together with a group homomorphism $\chi: P_s^{gp} = \ZZ^n \rightarrow \kk^\times$, determines a monomial map $Spec\text{ }k[t] \rightarrow Spec\text{ }\kk[P_s]$, defined by the ring homomorphism $u^*:\kk[P_s] \rightarrow \kk[t]$ given by $u^*(p) = \chi(p)t^{u(p)}$, for a choice of local coordinate $t$ of $0 \in \mathbb{A}^1$. There exists an \'etale neighborhood of $0 \in \mathbb{A}^1$, $U$, which factors through the \'etale morphism $O \rightarrow \kk[P_s]$. By picking $\chi$ generically, we can ensure that the resulting map $U \rightarrow X$ has image intersecting the Zariski dense subset $\pi_X(V) \subset X$. Indeed, we can identify a group homomorphism $\chi$ with a closed point of $Spec\text{ }\kk[P_s^{gp}] \subset Spec\text{ }\kk[P_s]$, and the corresponding morphism $\mathbb{A}^1 \rightarrow Spec\text{ }\kk[P_s]$ maps $1$ to $\chi \in Spec\text{ }\kk[P_s^{gp}]$. Moreover, since $O$ is connected, $j^{-1}(\mathbb{G}_m^n)$ is a Zariski dense open subset of $O$, hence has non-trivial intersection with $i^{-1}(\pi_X(V))$. It suffices to pick $\chi$ to be any geometric point in the image of this non-empty open subset. Letting $T = Spec\text{ } \mathcal{O}_{U,0}$, we have an induced morphism of schemes $T \rightarrow X$ such that the unique closed point of $T$ maps to $s\in X$. 

We wish to upgrade $T \rightarrow X$ to a morphism of log schemes. First, pick a morphism of monoids $g: P_s \rightarrow \mathbb{N}^n$, such that $g^t$ is injective, where $g^t:(\mathbb{N}^n)^\vee  \rightarrow  \overline{\mathcal{M}}_{X,s}^\vee$, let $u \in im(int((\mathbb{N}^n)^\vee))$, and $u' = (g^t)^{-1}(u)$. Note that this condition implies $g^{gp}: P^{gp}_s \rightarrow \mathbb{N}^n$ is an isomorphism. We equip $T$ with the log structure given by the global chart:

\[\alpha: \mathbb{N}^n \rightarrow \kk[t], \text{   } \alpha(e_i) = (\chi\circ(g^{gp})^{-1}(e_i))t^{u'(e_i)}.\]

By working in the \'etale chart around $s$, the morphism $T \rightarrow X$ is now easily seen to extend to a morphism of log schemes with $T$ equipped with the above log structure, special fiber mapping to $s$, and generic fiber mapping into $V \subset X$. Finally, after recalling we have replaced $\mathfrak{M}^{ev}(\mathcal{X},\beta)$ with the complement of all strata whose closures do not intersect the image of $\mathscr{M}(X,\beta) \rightarrow \mathfrak{M}^{ev}(\mathcal{X},\beta)$, the moduli stack $\mathfrak{M}^{ev}(\mathcal{X},\beta)$ is finite type. In particular, $\Sigma(\mathfrak{M}^{ev}(\mathcal{X},\beta))$ contains only finitely many distinct cones, and using the criterion of Lemma \ref{tran}, we may pick $u \in int(\sigma_s)$ and the initial monoid morphism $g$ so that the induced morphism $T \rightarrow X$ is transverse to the evaluation morphism $ev_{x_{out}}: \mathfrak{M}^{ev}(\mathcal{X},\beta) \rightarrow X$.

\textbf{Step 4: Construction of interpolating moduli space}

In this step, we will study how the degenerating point constraint induces a degeneration of the virtual geometry of the moduli space $\mathscr{M}(X,\beta)_x$, and study the tropical types associated to various virtual components of this degeneration. Similar arguments appear in \cite{scatt} Theorem $6.4$ Step $2$.

Recalling that we have morphisms of log stacks $ev_{x_{out}}:\mathscr{M}(X,\beta) \rightarrow X$ and $ev_{x_{out}}: \mathfrak{M}^{ev}(\mathcal{X},\beta) \rightarrow X$, we take a fine and saturated pullback to produce the moduli space $\mathscr{M}(X,\beta)_T := \mathscr{M}(X,\beta) \times_X^{fs} T$ and $\mathfrak{M}^{ev}(\mathcal{X},\beta)_T:= \mathfrak{M}^{ev}(\mathcal{X},\beta) \times_X^{fs} T$ of stable log curves of type $\beta$ with point constraint $T$. By Lemma \ref{flat}, flatness of $\mathfrak{M}^{ev}(\mathcal{X},\beta)_T$ over $T$ is guaranteed by the transversality condition assumed above. By taking the point constraint to be the generic point $\eta \rightarrow T$ instead, we also produce the moduli spaces $\mathscr{M}(X,\beta)_{\eta}$ and $\mathfrak{M}^{ev}(\mathcal{X},\beta)_{\eta}$. 
We summarize these constructions in the following fs cartesian diagram in fs log stacks:

 \[\begin{tikzcd}
 \mathscr{M}(X,\beta)_\eta \arrow{r}\arrow{d} & \mathscr{M}(X,\beta)_T \arrow{r}\arrow{d} & \mathscr{M}(X,\beta)\arrow{d} \\
 \mathfrak{M}^{ev}(\mathcal{X},\beta)_\eta \arrow{r}\arrow{d} & \mathfrak{M}^{ev}(\mathcal{X},\beta)_T \arrow{r}\arrow{d} & \mathfrak{M}^{ev}(\mathcal{X},\beta) \arrow{d} \\
 \eta \arrow{r} & T \arrow{r} & X
 \end{tikzcd}\]
 
After identifying $s$ with the unique closed point of $T$, we may similarly form the moduli spaces $\mathscr{M}(X,\beta)_s$ and $\mathfrak{M}^{ev}(\mathcal{X},\beta)_s$ as well, which will all be flat and integral over $Spec$ $\kk$ equipped with the pullback log structure from $T$. 

We now wish to understand the tropical type $\tau$ associated with a generic point $\xi$ of $\mathfrak{M}^{ev}(\mathcal{X},\beta)_s$. Specifically, we wish to show dim $\tau = n$. Since the point constraint $s \rightarrow X$ is chosen to be transverse to the evaluation map $ev: \mathfrak{M}^{ev}(X,\beta) \rightarrow X$, the pullback $\mathfrak{M}^{ev}(\mathcal{X},\beta)_s \rightarrow s$ is integral. Hence, by Lemma \ref{flat}, the log and standard fiber dimension of $\mathfrak{M}^{ev}(\mathcal{X},\beta)_s \rightarrow s$ coincide. In particular, after recalling definition of log fiber dimension in Definition \ref{fibdim} and noting $\xi$ is a generic point of $\mathfrak{M}^{ev}(\mathcal{X},\beta)_s$, for $Q_\xi$ the stalk of the ghost sheaf at the generic point, we must have:
\[\text{rank }Q_\xi^{gp} =\text{ rank }\overline{\mathcal{M}}_{s,s}^{gp} = \text{ rank }\overline{\mathcal{M}}_{X,s}^{gp} = \text{ dim }X.\]
Note by construction we have $Q_{\xi,\mathbb{R}}^\vee = \tau \times_{\sigma_s} (\mathbb{N}^n)^\vee_{\mathbb{R}}$ as a cone, with $\tau \rightarrow \sigma_s$ given by evaluation at $v_{out}$, the vertex adjacent to the leg corresponding to $x_{out}$, and $(\mathbb{N}^n)^\vee \rightarrow \sigma_s$ the transpose of the map of ghost sheaves of the point constraint defined above. This is the tropical moduli space of tropical curves of type $\tau$ with $v_{out}$ mapping into the image of $\mathbb{R}^n_{\ge 0} \rightarrow \sigma_s$. By construction of the logarithmic point constraint, we have $\mathbb{R}^n \rightarrow \sigma_s^{gp}$ is an isomorphism of real vector spaces, and the interior of the image of $\mathbb{R}_{\ge 0}^n$  is contained in the image of the interior of $ev_{v_{out}}: \tau \rightarrow \sigma_{s}$. Thus, there exists a point $q \in Q_{\xi,\mathbb{R}}^\vee$ mapping into the interior of $\tau$ and $\mathbb{R}_{\ge 0}^n$ under the projections, and the tangent space at such a point is given by $\tau^{gp} \times_{\sigma_s^{gp}} \mathbb{R}^n = \tau^{gp}$. In particular, dim $\tau = $ dim $h(\tau_{v_{out}}) = n$.

Now suppose that $v_{out}$ is at least $k+1$ valent in $G_\tau$, recalling that $k$ is the number of input marked points. Since dim $h(\tau_{v_{out}}) = n$, $v_{out}$ must be exactly $k+1$ valent. Indeed, note that since the genus of $G_\tau$ must be $0$, there exists a connected component of $G_\tau \setminus \{v_{out}\}$ which does not contain a leg of $G_\tau$. By Lemma \ref{GStlemma}(1), the tropical type associated with this connected component must have a unique leg spanning at most an $n-1$ dimensional cone. Since $h_\tau(v_{out})$ must be contained in a subcone of dimension bounded above by $n-1$, this would contradict $dim\text{ }h(\tau_{v_{out}}) = n$. 

By cutting the corresponding tropical type at all edges containing $v_{out}$, we produce $k$ broken line types. To see this, first note a degenerate case in which one of the edges adjacent to $v_{out}$ is in fact a leg, hence contains no vertex. In this case, we immediately produce a degenerate broken line type. Now suppose this is not the case for some edge containing $v_{out}$. By replacing this former edge with a punctured leg, which we call $L_{out}$, and taking the resulting tropical type whose dual graph does not contain $v_{out}$, this new tropical type $\tau_i$ has two legs, the new one we have identified, and $L_{in}$ which was a leg in the original tropical type. Moreover, as $v_{out}$ varies in an $n$ dimensional cone, dim $h(\tau_{i,L_{out}}) = n$, and dim $\tau_i \ge n-1$. We must have that dim $\tau_i < n$ however, or else by gluing the families, we would have dim $\tau > n$, a contradiction. Since $\tau_i$ inherits conditions of realizability and balancing from $\tau$, $\tau_i$ is indeed a broken line type.

\textbf{Step 5: Support of $\psi$ class contained in smooth locus}

In this step, we will derive sufficient restrictions on possible contributions to our enumerative problem to constrain them to the smooth locus used in the definition of the naive count $\eta(\textbf{P},\textbf{A})$. Before proceeding, we establish some useful language. Recall that $\overline{\psi}_{x_{out}}^{k-2} \in A^{k-2}(\overline{\mathscr{M}}_{0,k+1})$ is Poincar\'e dual to the class of a point $[pt] \in A_0(\overline{\mathscr{M}}_{0,k+1})$. Thus, to pick a  generic representative for $\overline{\psi}^{k-2}$ condition, we pick a generic point $[\overline{C}] \in \mathscr{M}_{0,k+1}$. Observe then that the class $[\mathscr{M}(X,\beta)_x]^{vir}\cap \overline{\psi}^{k-2}_{x_{out}} \in A_0(\mathscr{M}(X,\beta))$ is supported on the closed substack $\mathscr{M}(X,\beta)_x\times_{\overline{\mathscr{M}}_{0,k+1}} [\overline{C}]$ for a generic choice of $[\overline{C}] \in \mathscr{M}_{0,k+1}$. We say a map $C \rightarrow X$ over a scheme $S$ satisfies the generic $\overline{\psi}$ class condition if the corresponding morphism $S \rightarrow \mathscr{M}(X,\beta)_x$ factors through the closed substack $\mathscr{M}(X,\beta)_x\times_{\overline{\mathscr{M}}_{0,k+1}} [\overline{C}]$. Having established language, we will now show how the $\psi$ class condition restricts the possible tropical behavior of contributing curves:

\begin{lemma} 
Consider a curve $(C/\eta',f)$ associated to a point of $\mathscr{M}(X,\beta)_{\eta}$ satisfying a generic $\overline{\psi}_{out}$ class condition. Letting $(G,l)$ be the dual graph of $C_\eta$, and $v \in V(G)$ be in the spine of $G$, then $\pmb\sigma(v) = 0$. 
\end{lemma}

\begin{proof} 

Since $\mathscr{M}(X,\beta)_\eta \rightarrow \eta$ is finite type, any point $\eta' \rightarrow \mathscr{M}(X,\beta)_\eta$ specializes to a point finite \'etale over $\eta$. Thus, consider a point $\eta' \rightarrow \mathscr{M}(X,\beta)_{\eta}$, as considered in the statement of the lemma, and without loss of generality, we may assume $\eta'$ is a finite \'etale over $\eta$, which extends to a morphism $T' \rightarrow T$ associated with a finite extension of discrete valuation rings. By properness of $\mathscr{M}(X,\beta)_T \rightarrow T$, after possibly replacing $\eta'$ with the spectrum of a finite field extension, and $T'$ with the spectrum of a DVR in the field extension dominating it, there exists a curve $(\mathfrak{f},\mathfrak{C} \rightarrow T')$ whose generic fiber is $C$, which also satisfies the $\overline{\psi}$ class condition. Let $\tau'$ be the tropicalization of the basic log structure on $T'$. Note that $\tau'$ is the cone parameterizing the universal tropical family with tropical type that of the special fiber. Since the tropical type of the special fiber is marked by the tropical type of the generic fiber, there exists a face $\rho \subset \tau'$ which parameterizes the universal family of tropical curves with the tropical type of the generic fiber. We must show that all curves appearing in the family over $\rho$ satisfy the condition that all vertices lying in the convex hull of all legs map to the zero cone in $\Sigma(X)$.

 To do this, we show that for $t \in int(\tau')$, the convex hull of the legs of the tropical curve $\Gamma_t$ is a union of $k$ broken lines meeting at a point in the interior of the cone $\sigma_s$. To see this, note first that since $\pmb u (L_{out}) = 0$, we must have $\pmb\sigma(v_{out}) = \sigma_s$ for $v_{out} \in V(G_{\tau'})$ the vertex associated to the component containing $x_{out}$, so this component must necessarily be contracted to $s$ by the map. For the curve to satisfy the $\overline{\psi}$ class constraint, this component must have at least $k+1$ special points, including $x_{out}$. Thus, we must have $v_{out}$ is at least $k+1$ valent, with all connected components of $G_{\tau'} \setminus \{v_{out}\}$ containing a leg. In particular, the image of $(\mathfrak{C}_s,\mathfrak{f}_s)$ in $\mathfrak{M}^{ev}(X,\beta)_s$ must lie in a component with generic point having tropical type $\tau$ described at the end of step $4$, i.e. $\tau'$ is marked by $\tau$. Hence, $v_{out}$ is contained in $k+1$ edges which are contained in the convex hull of distinct legs $l$ and $v_{out}$. If $v_{out}$ was contained in an additional edge $e \in E(G_{\tau'})$ not given above, then after cutting the graph at $e$ and considering the resulting connected component not containing $v_{out}$, we would have a graph with $1$ leg. Since $\tau'$ is marked by $\tau$, all vertices in this graph would be marked by the class of a contracted curve, which would violate stability. Thus, $v_{out}$ is exactly $k+1$ valent, and $\tau'$ is a gluing of $k+1$ tropical types marked by broken line types, one of which is trivial associated with a leg of contact order $0$, at a vertex $v_{out}$. These types are produced by cutting the $k$ edges containing $v_{out}$. After enumerating the edges containing $v_{out}$, $e_1,\ldots,e_k$, we label the resulting tropical types marked by broken types by $\gamma_i$.

 Denote by $v_i \not= v_{out}$ the remaining vertex contained in $e_i$. Letting $cut_i: \tau' \rightarrow \gamma_i$ be the morphism of cones given by cutting at $e_i$ and projecting onto the component $\gamma_i$, note that $h_t(v) = h_{cut_i(t)}(v)$ if $v$ is contained in the connected component of $\Gamma_t\setminus\{v_{out}\}$ containing $v_i$. Since $\gamma_i$ is an expansion of a broken line type, by Corollary \ref{tlemma1}, $\textbf{u}(e_i) \notin h_{\gamma_i}(\gamma_{i,v_i})^{gp} = h_{\tau}(\tau_{v_i})^{gp}$. Now suppose $h_t(v_{out}) = 0$. Since $dim \text{ }h(\gamma_{i,v}) = n-1$, and $\textbf{u}(e_i) \notin h(\tau_{v_i})^{gp}$, we must have $h_t(v_i) = h_{cut(t)}(v_i) = 0$. Thus, by Corollary \ref{tlemma1}, all vertices in the convex hull of the legs of $\Gamma_{cut(t)}$ also map to $0$. Thus, for $t \in \tau$, if $h_t(v_{out}) = 0$, then for $v \in V(G_{\gamma_i})$, we have $h_t(v) = h_{cut_i(t)}(v) = 0$. Since $h_t(v_{out}) = 0$ for all $t \in \rho \subset \tau'$, the spine of the tropicalization of the generic fibre must have all vertices in the convex hull of legs mapping to the zero cone. 

\end{proof}

With the above lemma in mind, consider a curve $(C/\eta',f)$ of tropical type $\tau'$, associated with an $\eta'$ valued point contained in the support of a generic representative of the $\overline{\psi}$ class condition on $\mathscr{M}(X,\beta)_\eta$. If there exists $v \in V(G_{\tau'})$ such that $\pmb{\sigma}(v) \not= 0$, then $v$ cannot be in the spine of $G_{\tau'}$. But now consider the subtree of $G_{\tau'}$ given by the connected component containing $v$ resulting from cutting along the edge connecting $v_{out}$ and $v$. The existence of this tree contradicts the result from step $1$. Hence, we conclude that no component of a curve contributing to the log Gromov-Witten invariant of interest maps into $D$, that is, $(C/\eta',f)$ is induced by a morphism $\eta' \rightarrow \mathscr{M}^{tr}(U,\beta) \subset \mathscr{M}(X,\beta)$. Since the point constraint $\eta \rightarrow U$ factors through $V \subset U$, with $V$ as constructed in \cite{archmirror} Proposition $3.12$, all such curves are determined by morphisms $\eta \rightarrow \mathscr{M}^{sm}(U,\beta)$. Moreover, they all lie in the fiber over $(\eta,\overline{C}) \in U \times \mathscr{M}_{0,k+1}$, for generic $\overline{C}$. As derived in the discussion proceeding this proof, the deformations of such curves $\mathscr{M}^{sm}(U,\beta)$ are unobstructed. Hence, the log Gromov Witten invariant $N_{p_1,\ldots,p_k,0}^{\textbf{A}}$ is simply the degree of $(\mathscr{M}^{sm}(U,\beta)\times_U \eta) \times_{\mathscr{M}_{0,k+1}} [\overline{C}]$, which is simply the degree of the morphism $\mathscr{M}^{sm}(U,\beta) \times_{U\times \mathscr{M}_{0,k+1}} (\eta,[\overline{C}]) \rightarrow \eta$. Since $\Phi_{x_{out}}: \mathscr{M}^{sm}(U,\beta) \rightarrow U \times \mathscr{M}_{0,k+1}$ is finite \'etale by Proposition $3.13$ of \cite{archmirror}, and degrees of finite \'etale morphisms of varieties are invariant under base change, this is $\eta(p_1,\ldots,p_k,\textbf{A})$, as desired.

\end{proof}

\begin{example}
We continue with the running Example \ref{runex}. By \cite{int_mirror} Example $1.21$, the mirror algebra $R_{(X,D)}$ is isomorphic to 
\[R_{(X,D)} = \kk[NE(X)][\vartheta_{v_1},\vartheta_{v_2},\vartheta_{v_{1,1} + v_{2,1}},\vartheta_{v_{1,2}+v_{2,2}}]/(\vartheta_{v_1}\vartheta_{v_2} - \vartheta_{v_{1,1} + v_{2,1}} - \vartheta_{v_{1,2}+v_{2,2}}, \vartheta_{v_{1,2}+v_{2,2}}\vartheta_{v_{1,1} + v_{2,1}} - t\vartheta_{v_1}).\]

Consider the enumerative question of the number of lines in $\mathbb{P}^2$ going through a general point and intersecting $D_1$ and $D_2$ in one and two points respectively such that the cross ratio of the point mapping to the general point and the point of intersections with $D_1$ and $D_2$ satisfy a generic fixed cross-ratio. By the above theorem, together with Theorem $9.3$ of \cite{int_mirror}, this can be computed by computing the $\vartheta_0$ term of $\vartheta_{v_1}\vartheta_{v_2}^2 \in R_{(X,D)}$. To do this, we note the relation $\vartheta_{v_2}\vartheta_{v_{1,i}+v_{2,i}} = t^{[l]} + \vartheta_{v_{1,i} + 2v_{2,i}}$, which can be derived in a manner similar to Example $1.21$ of \cite{int_mirror}. Using this, observe:
\[tr(\vartheta_{v_2}^2\vartheta_{v_1}) = tr(\vartheta_{v_2}(\vartheta_{v_{1,1} + v_{2,1}} + \vartheta_{v_{1,2}+v_{2,2}})) = tr(\vartheta_{v_2}\vartheta_{v_{1,1} + v_{2,1}}) + tr(\vartheta_{v_2} \vartheta_{v_{1,2}+v_{2,2}}) = 2t^{[l]}.\]
Thus, there are two such lines.
\end{example}

We note that every affine variety $U$ admits an snc compactification by a globally generated divisor. Indeed, we may embed $U$ as a closed subset inside $\mathbb{A}^m$ for some $m$, and take the closure in $\mathbb{P}^m$, which we call $X'$. The compliment of $U$ is an ample divisor on $X'$, whose associated line bundle is given by the pullback of $\mathcal{O}_{\mathbb{P}^m}(1)$. By embedded resolution of singularities, there exists a sequence of blowups  $X \rightarrow X'$ supported away from $U$ such that $X \setminus U$ is snc.  We may pullback the divisor along this map, and produce a divisor whose associated line bundle is globally generated, and whose support is exactly $X \setminus U$.

In particular, every smooth affine log Calabi-Yau with a Zariski dense torus is compactified by a pair $(X,D)$ as above. Hence, the $2$ and $3$ input trace forms defined on the modules coincide. By the non-degeneracy of the trace form proven in \cite{archmirror} Section $18$, together with Proposition \ref{nondeg}, we conclude that the two definitions of $R_{(X,D)}$ coincide when considered as algebras over $\hat{S}_X$. Since the Keel-Yu product rule for theta functions has values in $S_X \subset \hat{S}_X$, so does the Gross-Siebert mirror, concluding the proof of Corollary \ref{mcr11}.

Corollary \ref{mcr11} implies that all structure constants used to define the non-archimedean and logarithmic mirrors coincide. After recalling how these structure constants were defined in \cite{archmirror} and \cite{int_mirror} respectively, we find that certain punctured log Gromov-Witten invariants coincide with certain counts of non-archimedean analytic disks:

\begin{corollary}
The punctured log Gromov-Witten invariants $\alpha_{p,q,\textbf{A}}^r$ for a target pair $(X,D)$ with $X\setminus D$ smooth connected affine containing a Zariski open torus, with $D$ the support of a nef divisor, are counts of non archimedean analytic disks which satisfy the toric tail condition in the sense of \cite{archmirror}.
\end{corollary}

Since the pullback of nef divisors are nef, the semipositivity condition is stable under further blowups, hence the collection of compactifications $(X,D)$ of $U$ for which Corollary \ref{mcr11} holds is cofinal among all compactifications of $U$ related by birational modifications with centers on $D$. 

Finally, we note the following birational invariance corollary, which follows from the birational invariance of the non-archimedean mirror given in \cite{archmirror} Proposition $17.3$:

\begin{corollary}
Let $(X,D)$ be an snc log Calabi-Yau with $X \setminus D$ a connected smooth affine log Calabi-Yau with Zariski dense torus, with $D$ the support of a nef divisor, and consider a birational morphism $(X',D') \rightarrow (X,D)$ between snc log Calabi-Yau pairs which restricts to an isomorphism $X'\setminus D' \rightarrow X \setminus D$. Then we have an isomorphism $R_{(X,D)} \cong R_{(X',D')} \otimes_{S_{X'}} S_X$, with $R_{(X,D)}$ the Gross-Siebert intrinsic mirror constructed in \cite{int_mirror}
\end{corollary}

The insensitivity of this algebra to different choices of compactifications is predicted under the expectation that the mirror algebra is equal to the degree $0$ symplectic cohomology of $U$, which itself is independent of compactification, see \cite{symcoh} section $4(b)$. For general log Calabi-Yau pairs $(X,D)$, assuming the birational morphism is log \'etale, this was proven in Corollary $1.6$ of \cite{pbirinv}

\section{Classical-Quantum period correspondence}

As a result of Theorem \ref{mthm11} and Theorem $1.2$ of \cite{archmirror}, given $p_1,\ldots,p_k \in B(\ZZ)$, we have a log Gromov-Witten interpretation of the coefficient of $\vartheta_0$ in $\vartheta_{p_1}\cdots \vartheta_{p_k}$ in the Gross-Siebert mirror algebra, proving Theorem \ref{mthm21} when $X\setminus D$ contains a Zariski dense torus. To remove this assumptions and replace it with Assumption $1.1$ of \cite{scatt}, we apply gluing arguments developed in \cite{toricglue} and used in \cite{scatt}.

We first formally define a tropical type which appeared in the proof of Theorem \ref{mthm11} above:

\begin{definition}[k-trace type]
A \emph{k-trace type} with inputs $p_1,\ldots p_k \in B(\mathbb{Z})$ is a type $\tau = (G,\pmb\sigma, \textbf{u})$ of tropical map to $\Sigma(X)$ such that:

\begin{enumerate}
\item G is a genus zero graph with $L(G) = \{L_1,\ldots,L_k,L_{out}\}$, with $G$ having a unique vertex $v_{out} \in V(G)$ of valence $k+1$, which we further require be adjacent to $L_{out}$, with $\pmb\sigma(L_i) ,\pmb\sigma(L_{out}) \in \mathscr{P}$ and $\textbf{u}(L_{out}) = 0$.

\item $\tau$ is realizable and balanced.

\item Let $h: \Gamma(G,l) \rightarrow \Sigma(X)$ be the corresponding universal family of tropical maps, and $\tau_{v_{out}} \in \Gamma(G,l)$ be the cone corresponding to the vertex $v_{out}$. Then $dim\text{ }\tau = dim\text{ }h(\tau_{v_{out}} )= n$.

\end{enumerate}
\end{definition}
For a $k$-trace type $\tau$, we define $k_\tau:= coker|ev_{v_{out}}(\tau^{gp}_{\NN}) \rightarrow \pmb\sigma(v_{out})_\NN^{gp}|$, analogous the corresponding definition $k_\gamma = coker|ev_{v_{out}}(\gamma^{gp}_{\NN}) \rightarrow \pmb\sigma(v_{out})_\NN^{gp}|$ for $\gamma$ a broken line type.

In proving the weak Frobenius structure theorem, we note that many arguments appearing in \cite{scatt} Section $6$ describing moduli stacks of punctured curves with marking by a product type carry over to describing moduli stacks of punctured curves with marking by a $k$-trace type. Indeed, when $k=2$, a $k$-trace type is a product type with the contact order of the output leg $0$. Mainly cosmetic changes to arguments in loc. cit. are necessary for us, but for completeness, we produce the necessary arguments in this section when slight changes are needed. 

First, let $s \in X$ be a zero stratum as in the proof of Theorem \ref{mthm11}, with $P_{\mathfrak{u}}:= \overline{\mathcal{M}}_{X,s}$, and denote by $P_{\mathfrak{u}}^{\vee} = \mathfrak{u} \in \Sigma(X)$ the maximal cone corresponding to the stratum $s$. Moreover, let $\pmb\tau$ be a decorated $k$-trace type, with vertex $v_{out} \in V(G_\tau)$ satisfying $\pmb\sigma(v_{out}) = \mathfrak{u}$, and consider the associated moduli space $\mathscr{M}(X,\pmb\tau)$. Letting $\psi_{x_{out}} \in A^1(\mathscr{M}(X,\pmb\tau))$ be the cotangent line associated with the marked point $x_{out}$, we note that we have the equality $\psi_{x_{out}} = Forget^*(\overline{\psi}_{x_{out}})$. Indeed, for $(C,f)$ a log map associated with a geometric point of $\mathscr{M}(X,\pmb\tau)$, the restriction of $f$ to the component $C'$ containing $x_{out}$ must factor through  $s \in X$. In particular, $f$ must contract $C'$, so for $C$ to be stable, $C'$ must have at least $3$ special points. Since $dim\text{ }h(\tau_{v_{out}}) = n$ by definition of a $k$-trace type, $h(\tau_{v_{out}})$ cannot be contained in a wall of the canonical scattering diagram of $(X,D)$. Thus, the connected components of $G_{\tau} \setminus \{v_{out}\}$ all contain legs not coming from edges containing $v_{out}$. Hence, all of the special points contained in $C'$ remain special after stabilizing the underlying curves, implying stabilizing the underlying curve $C$ does not contract $C'$, as required for $\psi_{x_{out}} = Forget^*(\overline{\psi}_{x_{out}})$. 

By simple modifications to the proof of Lemma $3.9$ of \cite{scatt}, the virtual dimension of $\mathscr{M}(X,\pmb\tau)$ is $k-2$. After enumerating the edges containing the unique $k+1$ valent vertex $v_{out}$ by $e_0 = L_{out},e_1,\ldots,e_k$, by cutting the edges $e_1,\ldots,e_k$, we produce a morphism $\delta: \mathscr{M}(X,\pmb\tau) \rightarrow \prod\mathscr{M}(X,\pmb\tau_i)$, where for $i>0$, $\pmb\tau_i$ is the tropical type produced by cutting along edges $e_i$ and taking the connected component which do not contain the vertex $v_{out}$, and $\pmb\tau_0$ is the resulting tropical type with a unique vertex $v_{out}$ with $k+1$ legs. It quickly follows from the definition of a $k$-trace type that $\pmb\tau_i$ is a decorated broken line type for $i>0$. The virtual dimensions of the moduli spaces $\mathscr{M}(X,\pmb\tau_i)$ are all $0$ for $i > 0$, and $k-2$ for $i = 0$. By \cite{punc} Theorem $C$, the cutting morphism $\delta$ is representable and finite. The following proposition gives a description of the class $\delta_*([\mathscr{M}(X,\tau)]^{vir})$. This closely follows \cite{scatt} Lemma $6.6$:

\begin{proposition}\label{glue}
$\delta_*([\mathscr{M}(X,\pmb\tau)]^{vir}) = \frac{\prod k_{\tau_i}}{k_\tau}\prod_i [\mathscr{M}(X,\pmb \tau_i)]^{vir} \times [\mathscr{M}(X,\pmb \tau_0)]^{vir}$

\end{proposition}

\begin{proof}
We first assume no type $\tau_i$ are trivial broken lines. Since $\tau_0$ must have a leg of contact order $0$ whose corresponding marked point maps to a zero stratum, all curves in $\mathscr{M}(X,\tau_0)$ must contract to the zero stratum. In particular, the decoration of $\tau_0$ must be trivial. Since we are working with genus $0$ stable curves, there are no complications due to the question of logarithmic enhancement, and the underlying stack of $\mathscr{M}(X,\tau_0)$ is $\overline{\mathscr{M}}_{0,k+1}$. We now apply the gluing formula of \cite{toricglue} to compute the push forward, by splitting at the edges of $\tau$ containing $v_{out}$. This gives the tropical types $\tau_0,\tau_1,\ldots,\tau_k$ described above. Note that by \cite{scatt} Proposition $1.3$ the gluing strata are isomorphic to strata in appropriate toric varieties. Thus, the main result of \cite{toricglue} applies. Using the notation of loc. cit. it is straightforward to check that the gluing map $\prod_{i} \tilde{\tau}_{i} \times \tilde{\tau}_0 \rightarrow \prod_{i} \pmb\sigma(e_i)^{gp}$ is surjective. Thus, the trivial displacement vector $\nu = 0$ is general for the type $\tau$. Moreover, the required dimension formula holds:

\[\text{dim }\tilde{\tau}_0 + \sum_i \text{dim }\tilde{\tau_i} = (n+k) + k((n-1)+1) = \text{dim }\tilde{\tau} + \sum_{e} \text{rk } \pmb\sigma(e_i)\]
Thus, the set of transverse types $\Delta(0)$ consists only of $\tau$. Applying the main result of \cite{toricglue} gives the result up to a splitting multiplicity calculation.

To compute the splitting multiplicity, after recalling the morphism of abelian groups $\epsilon_\tau: \tilde{\tau}_{0,\ZZ}^{gp} \times \prod_{i=1}^k \tilde{\tau}_{i,\ZZ}^{gp} \rightarrow (\sigma^k_\NN)^{gp}$ whose kernel is $\tilde{\tau}_{\NN}^{gp}$, consider the following diagram with exact rows and columns:

\[\begin{tikzcd}
& 0 \arrow{d} & 0 \arrow{d} & 0 \arrow{d} \\
0 \arrow{r} & \tilde{\tau}_{\mathbb{N}}^{gp} \arrow{r} \arrow{d} & \tilde{\tau}_{0,\mathbb{Z}}^{gp} \times \prod_i  \tilde{\tau}_{i,\ZZ}^{gp} \arrow{r}\arrow{d} & im(\epsilon_{\tau}) \arrow{r} \arrow{d} & 0\\
0 \arrow{r} & \sigma^{gp}_\NN \times \mathbb{Z}^k \arrow{r}\arrow{d} & (\sigma^{gp}_{\NN} \times \ZZ^k) \times (\sigma^{gp}_\NN)^k \arrow{r} \arrow{d} & (\sigma^{gp}_\NN)^k \arrow{r} \arrow{d} & 0\\
0 \arrow{r} & \ZZ/k_\tau\ZZ \arrow{r} \arrow{d} & \prod_i \ZZ/k_{\tau_i}\ZZ \arrow{r}\arrow{d} & coker (\epsilon_{\tau}) \arrow{r} \arrow{d} & 0 \\
& 0 & 0 & 0
\end{tikzcd}\]

The left column is the exact sequence defining $k_\tau$ , the middle column is the product of the exact sequences defining $k_{\tau_i}$, and the isomorphism $\tilde{\tau}_0^{gp} \rightarrow \sigma_{\NN}^{gp} \times \ZZ^k$. The exactness of the third row shows
\[\prod_i k_{\tau_i} = k_\tau |coker(\epsilon _{\tau})|\]
By the main theorem of \cite{toricglue}, the multiplicity we wish to compute is $|coker(\epsilon_{\tau})|$. Hence, we produce the desired formula in the case that all types $\tau_i$ are non trivial broken line types. 

If one of the tropical types $\tau_i$ are trivial broken line types, by cutting only along edges containing $v_{out}$, the same analysis as above but forgetting the trivial broken line types shows the formula holds in general.

\end{proof}

Using this decomposition of $\delta_*([\mathscr{M}(X,\pmb\tau)]^{vir})$, after recalling $N_{\pmb\tau_i} := \frac{1}{|Aut(\tau_i)|}deg[\mathscr{M}(X,\pmb\tau_i)]^{vir}$ from \cite{scatt} Section $3.3$, we observe the following corollary:

\begin{corollary}\label{pcalc}
$\frac{k_{\tau}}{|Aut(\tau)|}\int_{[\mathscr{M}(X,\tau)]^{vir}} \psi_{x_{out}}^{k-2} = \prod_{i\ge 1} k_{\tau_i}N_{\pmb\tau_i}$
\end{corollary}

\begin{proof}
We first observe $Forget: \mathscr{M}(X,\pmb\tau) \rightarrow \overline{\mathscr{M}}_{0,k+1}$ factors through $\pi_{\tau_0}\delta: \mathscr{M}(X,\pmb\tau) \rightarrow \mathscr{M}(X,\pmb\tau_0)$. Indeed, the morphism $Forget$ is induced by the stabilization of the universal family over $\mathscr{M}(X,\pmb\tau)$, which over a geometric point of $\mathscr{M}(X,\pmb\tau)$ contracts all components marked by a vertex not equal to $v_{out}$. Furthermore, any map associated with a geometric point of $ \mathscr{M}(X,\pmb\tau_0)$ is contracted to the zero stratum associated with $\pmb\sigma(v_{out})$, so the underlying family of log curve must be stable. Since the genus of the domain stable curve is $0$, standard arguments show that $C$ admits a log enhancement to give a basic map $C \rightarrow (Spec\text{ }\kk,\sigma(v_{out})^{\vee}_{\NN})$, hence the underlying stack of $\mathscr{M}(X,\pmb\tau_0)$ is isomorphic to $\overline{\mathscr{M}}_{0,k+1}$, with isomorphism given by the forgetful morphism. Additionally, the fiber of the obstruction sheaf over a curve $[(C,f)] \in \mathscr{M}(X,\pmb\tau_0)$ is $H^1(f^*T_X^{log}) = H^1(\mathcal{O}_{C})$, which is $0$ since we are working in genus $0$. Thus, $[\mathscr{M}(X,\pmb\tau_0)]^{vir} = [\mathscr{M}(X,\pmb\tau_0)]$. Furthermore, standard properties of $\psi$ classes on $\overline{\mathscr{M}}_{0,k+1}$ give $\overline{\psi}_{x_{out}}^{k-2} \cap [\overline{\mathscr{M}}_{0,k+1}] = [pt]$. We now use Proposition \ref{glue}, together with the fact that the cohomology class $\psi_{x_{out}}$ on $\mathscr{M}(X,\pmb\tau)$ is the pull back of the cohomology class $\overline{\psi}_{x_{out}}$ along the forgetful map, to calculate:

\[ \int_{[\mathscr{M}(X,\pmb\tau)]^{vir}} \psi_{x_{out}}^{k-2} = \frac{\prod_i k_{\tau_i}}{k_\tau}\int_{\prod_i[\mathscr{M}(X,\pmb\tau_i)]^{vir} \times [\mathscr{M}(X,\pmb\tau_0)]} \pi_{\tau_0}^* \psi_{x_{out}}^{k-2} =  \frac{\prod_i k_{\tau_i}deg[\mathscr{M}(X,\pmb\tau_i)]^{vir}}{k_{\tau}} .\]
The formula in the statement of the corollary follows from the definition of $N_{\pmb\tau_i}$ and observing $Aut(\tau) = \prod_{i\ge 0} Aut(\tau_i)$, since automorphisms in $Aut(\tau)$ must induce an automorphism of $G_\tau$ which fixes the legs, and each subgraph $G_{\tau_i} \subset G_{\tau}$ contains a unique leg from $G_{\tau}$. 
\end{proof}

Having expressed a certain log Gromov-Witten invariant in terms of a product of virtual counts of broken lines, we wish to relate these invariants with those considered in Theorem \ref{mthm21}. Using the interpolating moduli spaces constructed in step $4$ of the proof of Theorem \ref{mthm11}, by deformation invariance of the virtual fundamental class, we can compute the relevant degree by degenerating a generic point constraint $x \in X\setminus D$ to a zero stratum $s \in X$ of the pair $(X,D)$, and integrating the cohomology class $\psi_{x_{out}}^{k-2}$ instead against the virtual cycle $[\mathscr{M}(X,\beta)_s]^{vir}$. More precisely, we consider a point $x \rightarrow X\setminus D$ factoring through an \'etale neighborhood $O \rightarrow X$ considered in step $3$ of the proof of \ref{mthm11}, and by choosing the degeneration appropriately, we may further assume a further factorization through $C \rightarrow X$, with $C$ defined as in step $3$ of the proof of \ref{mthm11}. By two applications of deformation invariance, we have equalities of log Gromov-Witten invariants after imposing a point constraint $\eta$, for $\eta$ the generic point of $C$, and after imposing a logarithmic point constraint $s$. 

After degenerating the point constraint, note that we do not need to assume the boundary supports a nef divisor for the equality $\psi_{x_{out}} = Forget^*(\overline{\psi}_{x_{out}})$ in $A^1(\mathscr{M}(X,\beta)_s)$. The argument for this is esentially the same as the one showing the analogous result for $\mathscr{M}(X,\tau)$ for $\tau$ a $k$-trace type. The only different input needed for this argument is that for $(C,f)$ a log map associated with a geometric point of $\mathscr{M}(X,\beta)_s$, with the tropicalization of the base being the cone $\tau'$, we have $dim\text{ }h(\tau'_{v_{out}}) = n$ by the tropicalization of the point constraint. Thus, we have the following equality of log Gromov-Witten invariants:

\begin{equation}
 \int_{[\mathscr{M}(X,\beta)_x]^{vir}} \psi_{x_{out}}^{k-2} = \int_{[\mathscr{M}(X,\beta)_s]^{vir}}\psi_{x_{out}}^{k-2} = \sum_{\xi} m_\xi\int_{\epsilon^![\mathfrak{M}^{ev}_{\xi,red}]} Forget^*\overline{\psi}_{x_{out}}^{k-2}
\end{equation}

The above sum varies over generic points of $\xi \in \mathfrak{M}^{ev}(\mathcal{X},\beta)_s$, with $\mathfrak{M}^{ev}_{\xi,red}$ the reduced stack associated to a component, and $m_\xi$ is the multiplicity of the associated component. The same analysis as in \cite{scatt} Section $6$ step $3$ shows $m_\xi = m_\tau = coker|\mathbb{Z}^n \rightarrow Q_\xi^{gp}|$, for $Q_\xi$ the stalk of the ghost sheaf at the generic point $\xi$. As shown in the proof of Theorem \ref{mthm11}, $Forget^*\overline{\psi}^{k-2}$ for a generic choice of representative is only supported on curves with image in $\mathfrak{M}^{ev}(\mathcal{X},\beta)_s$ contained in components with an associated tropical type a $k$-trace type. Thus, to compute the integral of interest, it suffices to pull back only components of $\mathfrak{M}^{ev}(\mathcal{X},\beta)_s$ with associated tropical type a $k$-trace type.

Given a $k$-trace type $\tau$, we have a moduli stack $\mathfrak{M}^{ev}(\mathcal{X},\tau)_s := \mathfrak{M}^{ev}(\mathcal{X},\tau) \times_{X} s$. Moreover, we have a morphism $i: \mathfrak{M}^{ev}(\mathcal{X},\tau)_s \rightarrow \mathfrak{M}^{ev}(\mathcal{X},\beta)_s$ and diagram with cartesian squares: 

\[\begin{tikzcd}
\coprod_\tau \mathscr{M}(X,\pmb\tau) \arrow[d,"\epsilon_\tau"] &  \coprod_{\pmb\tau = (\tau,A)} \mathscr{M}(X,\pmb\tau)_s \arrow{l} \arrow{r} \arrow[d,"\epsilon_{\tau,s}"] & \mathscr{M}(X,\beta)_s \arrow[d,"\epsilon_s"] \\
\coprod_\tau \mathfrak{M}^{ev}(\mathcal{X},\tau) & \coprod_\tau \mathfrak{M}^{ev}(\mathcal{X},\tau)_s \arrow{r} \arrow{l} & \mathfrak{M}^{ev}(\mathcal{X},\beta)_s
\end{tikzcd}\]

Observe that the morphism $\mathfrak{M}^{ev}(\mathcal{X},\tau)_s \rightarrow \mathfrak{M}^{ev}(\mathcal{X},\beta)_s$ has degree $|Aut(\tau)|$, since the morphism forgets the marking by the tropical type $\tau$. After replacing $\mathfrak{M}^{ev}(\mathcal{X},\tau)_s$ with its reduction $\mathfrak{M}^{ev}(\mathcal{X},\tau)_{s,red}$, then after restricting the sum to be over components with generic points $\xi$ having tropical type a $k$-trace type, we have:

\begin{equation}\label{eq11}
\sum_\xi m_\xi [\mathfrak{M}^{ev}_{\xi,red}] = \sum_\tau \frac{m_{\tau}}{|Aut(\tau)|}i_*[\mathfrak{M}^{ev}(\mathcal{X},\tau)_{s,red}].
\end{equation}

We also observe that since $\tau$ is realizable, $\mathfrak{M}(\mathcal{X},\tau)$ is reduced by Proposition $3.28$ of \cite{punc}, and so is $\mathfrak{M}^{ev}(\mathcal{X},\tau)$. Thus, the morphism $\mathfrak{M}^{ev}(\mathcal{X},\tau)_s \rightarrow \mathfrak{M}^{ev}(\mathcal{X},\tau)$ induces a morphism $\mathfrak{M}^{ev}(\mathcal{X},\tau)_{s,red} \rightarrow \mathfrak{M}^{ev}(\mathcal{X},\tau)$. To relate the invariant of interest to an integral against $[\mathscr{M}(X,\tau)]^{vir}$, we need to calculate the degree of $\mathfrak{M}^{ev}(\mathcal{X},\tau)_{s,red} \rightarrow \mathfrak{M}^{ev}(\mathcal{X},\tau)$. This calculation is as in \cite{scatt} Section $6$ step $4$, and is given by $|coker(P_\mathfrak{u}^{gp} \rightarrow Q^{gp} \oplus \mathbb{Z}^n)_{tor}|$. Furthermore, by definition of the fine and saturated fiber product, we have $Q_\xi^{gp} =  coker(P_\mathfrak{u}^{gp} \rightarrow Q_{\tau}^{gp} \oplus \mathbb{Z}^n)/tor$. Now consider the following morphisms of exact sequences of abelian groups:

\[
\begin{tikzcd}
0\arrow{r}  & P_{\mathfrak{u}}^{gp} \arrow{r} \arrow{d} & P_{\mathfrak{u}}^{gp} \oplus \mathbb{Z}^n \arrow{r} \arrow{d} & \mathbb{Z}^n \arrow{r} \arrow{d} & 0 \\
0 \arrow{r} & Q_\tau^{gp} \oplus \mathbb{Z}^n \arrow{r} & Q_\tau^{gp} \oplus \mathbb{Z}^n \arrow{r} & 0 \arrow{r} & 0
\end{tikzcd}
\]

The homomorphism $P^{gp}_{\mathfrak{u}} \rightarrow Q_\tau^{gp}$ is injective since the tropical modulus of a broken line $h_t: \Gamma_t \rightarrow \Sigma(X)$ is determined by $h_t(v_{out})$, for $v_{out}$ the vertex adjacent to the leg contained in $\mathfrak{u} \in \Sigma(X)$. Thus, $P^{gp}_{\mathfrak{u}} \oplus \mathbb{Z}^n \rightarrow Q_\tau^{gp} \oplus \mathbb{Z}^n$ is injective, and an application of the snake lemma gives the exact sequence:

\[0 \rightarrow \mathbb{Z}^n \rightarrow coker(P_{\mathfrak{u}}^{gp} \rightarrow Q_\tau^{gp} \oplus \mathbb{Z}^n) \rightarrow coker(P_{\mathfrak{u}}^{gp} \rightarrow Q_\tau^{gp}) \rightarrow 0.\]

We also have the following morphism of short exact sequences:

\[
\begin{tikzcd}
0\arrow{r}  & \mathbb{Z}^n \arrow{r} \arrow{d} & coker(P_{\mathfrak{u}}^{gp} \rightarrow Q_\tau^{gp} \oplus \mathbb{Z}^n) \arrow{r} \arrow{d} & coker(P_{\mathfrak{u}}^{gp} \rightarrow Q_\tau^{gp}) \arrow{r} \arrow{d} & 0 \\
0 \arrow{r} & \mathbb{Z}^n \arrow{r} & Q_\xi^{gp} \arrow{r} & coker(\mathbb{Z}^n \rightarrow Q_\xi^{gp})  \arrow{r} & 0
\end{tikzcd}
\]

The middle vertical morphism is the torsion quotient, and the rightmost vertical morphism is induced by the inclusion $Q_\tau^{gp} \rightarrow Q_\tau^{gp} \oplus \mathbb{Z}^n$. Another application of the snake lemma produces an isomorphism:
\[coker(P_{\mathfrak{u}}^{gp} \rightarrow Q_\tau^{gp} \oplus \mathbb{Z}^n)_{tor} \cong ker(coker(P_{\mathfrak{u}}^{gp} \rightarrow Q_\tau^{gp}) \rightarrow coker(\mathbb{Z}^n \rightarrow Q_\xi)).\]
 Using this isomorphism, we find:

\[\begin{split}
|coker(P_\mathfrak{u}^{gp} \rightarrow Q^{gp} \oplus \mathbb{Z}^n)_{tor}||coker(\mathbb{Z}^n \rightarrow Q_\xi)| &= |coker(P_\mathfrak{u}^{gp} \rightarrow Q_\tau^{gp})|\\
&= |coker(\tau^{gp}_\ZZ \rightarrow \mathfrak{u}_\ZZ^{gp})|\\
&= k_\tau
.\end{split}\]

Recalling that $m_\tau = |coker(\mathbb{Z}^n \rightarrow Q_\xi^{gp})|$, we thus find:

\begin{equation}\label{eq12}
k_\tau [\mathfrak{M}^{ev}(\mathcal{X},\tau)] = m_\tau k_{*,red}[ \mathfrak{M}^{ev}(\mathcal{X},\tau)_{s,red}] .
\end{equation}

By using the equations in (\ref{eq11}) and (\ref{eq12}), as well as the pull-push formula of Proposition $4.1$ of \cite{vpull} for the second and third equality, we express the left hand side of claimed equality of Theorem \ref{mthm21} as: 

\[\begin{split}
\sum_{\beta = (\textbf{A},\underline{\beta})} t^\textbf{A} \int_{[\mathscr{M}(X,\beta)]_s^{vir}} \psi_{x_{out}}^{k-2} &= \sum_{\textbf{A} \in NE(X)} t^\textbf{A}  \sum_{\xi} m_\xi \int_{\epsilon^![\mathfrak{M}^{ev}_{\xi,red}]} \psi_{x_{out}}^{k-2}\\
&= \sum_{\textbf{A} \in NE(X)} t^\textbf{A} \sum_{\pmb\tau} \frac{m_\tau}{Aut(\tau)}  \int_{\epsilon_{\pmb\tau,s}^![\mathfrak{M}^{ev}(\mathcal{X},\tau)_{s,red}]} \psi_{x_{out}}^{k-2}\\
&= \sum_{\textbf{A} \in NE(X)} t^\textbf{A}  \sum_{\pmb\tau} \frac{k_\tau}{Aut(\tau)}  \int_{\epsilon_{\pmb\tau}^![\mathfrak{M}^{ev}(\mathcal{X},\tau)]} \psi_{x_{out}}^{k-2}\\
&= \sum_{\textbf{A} \in NE(X)} t^\textbf{A} \sum_{\pmb\tau}  \prod_i k_{\tau_i}N_{\pmb\tau_i}.
\end{split} \]

The sum over $\pmb\tau$ above is over decorated $k$-trace types with total curve class $\textbf{A}$, and the last equality follows by Corollary \ref{pcalc}. Note, $\langle \vartheta_{p_1} \ldots \vartheta_{p_k}\rangle$ by definition equals the coefficient in $S_X$ in front of the $\vartheta_0$ term of $\vartheta_{p_1}\cdots \vartheta_{p_k}$. Moreover, as follows from step $4$ from the proof of Theorem \ref{mthm11}, every $k$-trace type $\tau$ is uniquely determined by a choice of $k$ broken line types $\tau_i$ such that the sum of the slopes of outgoing legs $L_{i,out}$ equals $0$. Using the associativity of the product rule for theta functions, together with the expression for products of theta functions derived in \cite{scatt} Section $6$, it is now straightforward to show the desired equality of Theorem \ref{mthm21}.

Now consider the case of $(X,D)$ an snc Fano pair. In \cite{fanoperiod}, Mandel states a conjecture, which in \cite{fanoperiod} is referred to as Conjecture $1.4$, and shows that if this conjectures holds, there is an equality between certain power series encoding quantum mirror periods of $(X,D)$, and the classical periods of its Landau-Ginzburg mirror $\check{X}$. We recall the conjecture below:

\begin{conjecture}[\cite{fanoperiod} Conjecture 1.4]\label{ncount}
For $\vartheta_{p_1},\ldots,\vartheta_{p_k}$ elements of the $S_X$-module basis for $R_{(X,D)}$, $\textbf{p} = (p_1,\ldots,p_k)$, and $\tilde{X} \rightarrow X$ a log \'etale modification such that the contact orders are either $0$ or divisorial, let $N_{\textbf{A}}^{naive}(\bold{p})$ be the number of marked maps $f: (\mathbb{P}^1,x_1,\ldots,x_k,x_{out},s) \rightarrow \tilde{X}$, with $x_1,\ldots,x_k,x_{out},s$ a general fixed configuration of distinct points on $\mathbb{P}^1$, $f(\mathbb{P}^1)$ meeting $D$ with contact order $p_i$ at the marked point $x_i$, and the marked point $x_{out}$ mapping to a fixed general point in $X \setminus D$, and $f_*[\mathbb{P}^1] = \textbf{A}$. Then $N_{\textbf{A}}^{naive}(p_1,\ldots,p_k)$ is finite, as is the sum:
\begin{equation}
\langle \vartheta_{p_1},\ldots,\vartheta_{p_k}\rangle^{naive} := \sum_{\textbf{A} \in NE(\tilde{X})} t^{\pi_*(\textbf{A})} N_\textbf{A}^{naive}(\textbf{p}) \in \mathbb{Z}[NE(X)].
\end{equation}
Extend $\langle \cdot\rangle^{naive}$ to define the $S_X$ multilinear $k$-point function:
\[ \langle \cdot \rangle : R_{(X,D)}^k \rightarrow S_X\]
Then there exists a product $\star$ on $R_{(X,D)}$ making it into a commutative associative $S_X$ algebra with identity $\vartheta_0$ such that:

\[\langle \vartheta_{p_1},\ldots,\vartheta_{p_k}\rangle^{naive} = \langle \vartheta_{p_1}\star \cdots \star \vartheta_{p_k}\rangle^{naive}\]
for all $r$ tuples $\vartheta_{p_1},\ldots,\vartheta_{p_k}$, $r\ge 1$. Moreover $\langle \vartheta_p\rangle^{naive}$ is $0$ if $p \not= 0$, and $1$ otherwise, so the righthand side above is given by taking the $\vartheta_0$ coefficient of the product.
\end{conjecture}

This conjecture, with the additional hypotheses that allow for the construction of a canonical scattering diagram, can now be derived from the intrinsic mirror construction, as well as Theorems \ref{mthm11} and \ref{mthm21} above. First, by Theorem \ref{mthm11}, we know $N^{\textbf{A}}_{0,0,p_1,\ldots,p_k} = N^{naive}_{\textbf{A}}(\textbf{p})$. By the result of Theorem \ref{mthm21} we know that the log Gromov-Witten invariant  $N^{\textbf{A}}_{0,0,p_1,\ldots,p_k}$ equals the $\vartheta_0$ term of $\vartheta_0 \cdot \vartheta_{p_1}\cdots \vartheta_{p_k}$. Now note that by Theorem \ref{mthm11}, $\langle \vartheta_p \rangle^{naive} = \sum_{A \in NE(X)} N_{p,0,0} z^A$, i.e. the $\vartheta_0$ term of $\vartheta_0 \cdot \vartheta_p  = \vartheta_p$. Since $\vartheta_0$ is the identity in $R_{(X,D)}$, we have $\langle \vartheta_p \rangle^{naive}$ is equal to $1$ if $p = 0$ and $0$ otherwise. Combining these results therefore gives Conjecture \ref{ncount} for Fano pairs satisfying Assumption $1.1$ of \cite{scatt}. By Theorem $1.12$ from \cite{fanoperiod}, we derive Corollary \ref{mcr2} from the introduction.

\begin{example}
We return to our running Example \ref{runex}. Since the regularized quantum period only depends on the ambient Fano variety, and not on a choice of representative for the anti-canonical divisor, we can compute $\hat{G}_{\mathbb{P}^2}$ using the standard toric boundary. It is straightforward to calculate in this case that $\hat{G}_X = \sum_{k = 0}^\infty \frac{3k!}{(k!)^3}t^{k[l]}$, which after multiplying the coefficient in front of $t^{k[l]}$ by $\frac{1}{(-K_X \cdot k[l])!} = \frac{1}{(3k)!}$ recovers the well known quantum periods of $\mathbb{P}^2$, see \cite{Fanmir} Example $4.5$. Working with the anti-canonical representative given by the generic line and conic from before, we have $W = \vartheta_{v_1} + \vartheta_{v_2}$, and:

\[\pi_W = \sum_{k=0}^\infty tr(W^k) = \sum_{k=0}^\infty  tr(\sum_{j = 0}^k {k \choose j} \vartheta_{v_1}^j \vartheta_{v_2}^{k-j})\]
Using arguments similar to those used to derive equations for the mirror family, we have $tr(\vartheta_{v_1}^i\vartheta_{v_2}^j)$ can only be non zero when $j = 2i$. Thus, we can re-express the previous formula as:
\[\pi_W = \sum_{k=0}^\infty {3k \choose k} tr(\vartheta_{v_2}^{2k}\vartheta_{v_1}^k)\]
Since we know $\pi_W = \hat{G}_X$, we must have:
\[{3k \choose k} tr(\vartheta_{v_2}^{2k}\vartheta_{v_1}^k) =  \frac{3k!}{(k!)^3}t^{k[l]}\]
This implies that $tr(\vartheta_{v_2}^{2k}\vartheta_{v_1}^k) = {2k \choose k}t^{k[l]}$. Thus, by Theorem \ref{mthm11}, there are ${2k \choose k}$ morphisms $f: \mathbb{P}^1 \rightarrow \mathbb{P}^2$ with $f_*([\mathbb{P}^1]) = k[l]$ going through a generic point in $\mathbb{P}^2$, and intersecting the line, conic, and generic point in a fixed generic configuration of points on the domain. This agrees with computations done in \cite{scatloo}.
\end{example}

\begin{remark}
In \cite{RQC_no_log}, Tseng and You derive a similar result for their construction of an orbifold version of the mirror algebra, defined by restriction to the degree zero part of the orbifold relative quantum cohomology ring. Analogous to how the WDVV relation in orbifold GW theory is used to derive the associativity of the orbifold mirror algebra, loc. cit. uses topological recursion relations to deduce the weak Frobenius structure theorem. For both questions in the log setting, direct gluing arguments are used, as analogues of these relations in the log setting are not known. 
\end{remark}

\nocite{*}
\bibliographystyle{amsalpha}
\bibliography{Mirrorcomp}

\end{document}